\documentclass[10 pt,reqno]{amsart}

\usepackage[normalem]{ulem}%pour barrer du texte

\usepackage{amsmath,amssymb,color}
\usepackage{amsfonts, amscd, epsfig, amsmath, amssymb,enumerate}
\usepackage{graphicx}
\usepackage{graphics}
\usepackage{color}
\usepackage{mathrsfs}
\usepackage{todonotes}
\usepackage{soul}
\usepackage[displaymath, mathlines,pagewise]{lineno}
 \usepackage[usenames,dvipsnames]{pstricks}
 \usepackage{pstricks-add}
 \usepackage{epsfig}
 \usepackage{pst-grad} % For gradients
 \usepackage{pst-plot} % For axes
 \usepackage[space]{grffile} % For spaces in paths
 \usepackage{etoolbox} % For spaces in paths
\usepackage[margin=3cm]{geometry}
 \makeatletter % For spaces in paths
 \patchcmd\Gread@eps{\@inputcheck#1 }{\@inputcheck"#1"\relax}{}{}
 \makeatother

\newtheorem{theorem}{Theorem}[section]
\newtheorem{lemma}[theorem]{Lemma}
\newtheorem{proposition}[theorem]{Proposition}

\newtheorem{remark}[theorem]{Remark}

\newtheorem{assumption}{Assumption}[section]

\usepackage{tikz}
\usetikzlibrary{backgrounds}
\usetikzlibrary{patterns,fadings}
\usetikzlibrary{arrows,decorations.pathmorphing}
\usetikzlibrary{decorations}
\usetikzlibrary{calc}
\usetikzlibrary{shapes.misc}

\usepackage{setspace}
\definecolor{light-gray}{gray}{0.95}
\usepackage{float}
\usepackage[colorlinks=true,linkcolor=blue,citecolor=magenta]{hyperref}
\def\centerarc[#1](#2)(#3:#4:#5){\draw[#1] ($(#2)+({#5*cos(#3)},{#5*sin(#3)})$) arc (#3:#4:#5);}

\newcommand{\vertiii}[1]{{\left\vert\kern-0.25ex\left\vert\kern-0.25ex\left\vert #1 
    \right\vert\kern-0.25ex\right\vert\kern-0.25ex\right\vert}}

\numberwithin{equation}{section}
\numberwithin{figure}{section}

\newcommand{\<}{\big\langle}
\renewcommand{\>}{\big\rangle}

\renewcommand{\epsilon}{\varepsilon}

\newcommand{\R}{\mathbb R}
\newcommand{\Z}{\mathbb Z}
\newcommand{\N}{\mathbb N}
\renewcommand{\P}{\mathbb P}
\newcommand{\T}{\mathbb T}
\newcommand{\E}{\mathbb E}

\allowdisplaybreaks %%Pour \UTF{00E9}viter les grands espacements verticaux.

\title{Moderate deviation principles for a reaction diffusion model in non-equilibrium}

\author{Linjie Zhao}
\address{(Linjie Zhao) School of Mathematics and Statistics, and Hubei Key Laboratory of Engineering Modeling and Scientific Computing, Huazhong University of Science and Technology, Wuhan 430074, China.}
\email{linjie\_zhao@hust.edu.cn}

\thanks{\textbf{Acknowledgments.} The project is supported by the Fundamental Research Funds for the Central Universities in China, and by the National Natural Science Foundation of China with grant numbers 11971038 and 12371142.}

\begin{document}

\maketitle

\begin{abstract}
We study moderate deviations  from hydrodynamic limits of a reaction diffusion model.  The process is defined as the superposition of the symmetric exclusion process with a Glauber dynamics. When the process starts from a product measure with a constant density, which is a non-equilibrium measure for the process, we prove that the re-scaled density fluctuation field satisfies the moderate deviation principle. Our proof relies on the  so-called main lemma developed by Jara and Menezes \cite{jara2018non,jaram20nonequilireaction}.  
\end{abstract}

\section{Introduction}

The theory of hydrodynamic limits concerns about deriving macroscopic laws directly from microscopic dynamics.  Fluctuations and large deviations from hydrodynamic limits have been investigated intensively, see \cite{klscaling}.  The regime of moderate deviations is between fluctuations and large deviations, and as far as we know, moderate deviations from hydrodynamic limits was only proved for a few models.  Gao and Quastel in  \cite{gao2003moderate} first  investigated moderate deviations from hydrodynamic limits  of the symmetric simple exclusion process (SSEP).  Since then, moderate deviation principles (MDP) from hydrodynamic limits was extended to  other interacting particle systems, including a special Ginzburg-Landau model \cite{wang2006moderate}, the SSEP with a slow bond  in one dimension \cite{xue2021moderate}, and the weakly asymmetric  simple exclusion process \cite{zhao2024moderate}.  The main difficulty is to prove a super-exponential version of the Boltzmann-Gibbs principle. Formally speaking, the Boltzmann-Gibbs principle states that one could replace the non-conserved quantities of the dynamics by its density fluctuation field under the correct time scaling \cite[Chapter 11]{klscaling};  while when considering moderate deviations, one needs to prove such replacement is super-exponentially small.

We underline that  the above results are all concentrated on the stationary case, \emph{i.e.}, when the process starts from its stationary measure.   It  remains challenging to prove MDP from hydrodynamic limits  when the process is not stationary as one needs to prove a non-equilibrium version of the super-exponential  Boltzmann-Gibbs principle. As far as we know, non-equilibrium MDP from hydrodynamic limits was only proved for the SSEP in one dimension \cite{xue2024nonequilibrium} very recently.  The proof uses large deviation estimates for the SSEP. Moreover, since the SSEP is linear,  one does not need to prove the super-exponential  Boltzmann-Gibbs principle.

The aim of this paper is to investigate non-equilibrium moderate deviations  for other interacting particle systems.  Precisely, we study the reaction diffusion model in \cite{gonccalves2024clt}. The process is composed of the SSEP dynamics and a Glauber dynamics.  In the SSEP, particles perform symmetric random walks on the discrete torus subject to the constraint that there is at most one particle at each site.  In the Glauber dynamics, particles are created or destroyed at each site according to some rate.  It was proved in \cite{de1986reaction} that the hydrodynamic limits of this model are governed by the reaction diffusion equation. Non-equilibrium fluctuations were proved in \cite{jaram20nonequilireaction} by using the improved relative entropy method. For dynamical large deviations of this model, we refer the readers to \cite{landim2018hydrostatics,jona1993large}.  In this article, we proved moderate deviation principles for the rescaled density fluctuation field when the process starts from a special product measure, which is not invariant for the dynamics, see Theorem \ref{thm mdp}. 

As in the theory of large and moderate deviations,  for the upper bound, we investigate a super-exponential martingale related to the process, and for the lower bound, we need to study hydrodynamic limits for a perturbed dynamics, see \cite{klscaling} for example.  Compared to the proof in \cite{xue2024nonequilibrium}, we do need to prove the non-equilibrium super-exponential version of the  Boltzmann-Gibbs principle, see Lemmas \ref{lem: super-esponential estimate 1} and \ref{lem: super-esponential estimate 2}, which is the main novelty of the paper. We prove them directly by using the so-called main lemma in \cite{gonccalves2024clt},  which is crucial in the improved relative entropy method developed by Jara and Menezes \cite{jara2018non,jaram20nonequilireaction}.

Since the process is irreducible, it has a unique invariant measure, which is called non-equilibrium stationary state (NESS) in the literature.  Fluctuations and large deviations for the NESS were studied in \cite{gonccalves2024clt,farfan2019static}. We leave it as a future work to study the MDP for the NESS.

The rest of the paper is organized as follows.  In Section \ref{sec: model} we state the model and results rigorously. The MDP upper and lower bound are proved respectively in Sections \ref{sec: upper bound} and \ref{sec: lower bound}.

\subsection{Notation} We denote $\mathbb{N} = \{1,2,\ldots\}$. For $x,y \in \Z^d$ or $\T_n^d$, write $x \sim y$ if $\sum_{i=1}^d |x_i - y_i| = 1$.
For a probability measure $\mu$ on some space, we denote by $E_\mu [\,\cdot\,]$ the expectation with respect to the measure $\mu$.  For two sequences of positive real numbers $\{a_n\}_{n \geq 1}$ and $\{b_n\}_{n \geq 1}$, write $a_n \ll b_n$ or $a_n = o(b_n)$ if $\lim_{n \rightarrow \infty} a_n / b_n = 0.$

\section{Model and results}\label{sec: model}

The state space of the model is $\Omega_n := \{0,1\}^{\T_n^d}$, where $\T_n^d := \Z^d / (n \Z^d)$ is the $d$-dimensional discrete torus with size $n \in \N$.  The generator of the process $\eta (t) \equiv \eta^n (t)$ is $L_n = n^2 L_n^{\rm ex} + L_n^{\rm r}$, where for any function $f: \Omega_n \rightarrow \R$,
\begin{align*}
	L_n^{\rm ex} f (\eta) &= \sum_{x \in \T_n^d} \sum_{i=1}^d [f(\eta^{x,x+e_i}) - f(\eta)],\\
	L_n^{\rm r} f (\eta) &= \sum_{x \in \T_n^d} c_x (\eta) [f(\eta^{x}) - f(\eta)].
\end{align*}
Above, for $x,y \in \T^d_n$, $\eta^{x,y}$ is the configuration obtained from $\eta$ by swapping the values of $\eta_x$ and $\eta_y$,  and $\eta^x$ is the one by flipping the value of $\eta_x$, \emph{i.e.},
\[\eta^{x,y}_z = \begin{cases}
	\eta_x, \quad &\text{if } z=y,\\ 
	\eta_y, \quad &\text{if } z=x,\\
	\eta_z, \quad &\text{otherwise.}    
\end{cases} \qquad \eta^{x}_z = \begin{cases}
1- \eta_x, \quad &\text{if } z=x,\\ 
\eta_z, \quad &\text{otherwise.}    
\end{cases}\]
The flipping rate $c_x (\eta)$ is defined as
\[c_x (\eta) = \Big(a + \frac{\lambda}{2d} \sum_{y \sim x} \eta_y  \Big) (1-\eta_x) + b \eta_x,\]
where $a,b > 0$ and $\lambda > -a$ are given parameters.  

For $\rho \in (0,1)$, let $\nu^n_\rho$ be the product measure on $\Omega_n$ with constant particle density $\rho$,
\[\nu^n_\rho (\eta_x = 1, \, \forall x \in A) = \rho^{|A|}, \quad \forall A \subset \T_n^d.\]
Direct calculations show that
\[F(\rho) := E_{\nu^n_\rho} [c_x (\eta) (1-2\eta_x)] = (a+\lambda \rho) (1-\rho) - b \rho.\]
Since $F$ is a quadratic polynomial such that $F(0) = a > 0, F(1) = -b < 0$,  there exists a unique point $\rho_* \in (0,1)$ such that $F(\rho_*) = 0$.

Let $\mathcal{D} (\T^d)$ be the space of smooth functions on $\T^d$ and  let $\mathcal{D}^\prime (\T^d)$ be its topological dual. Fix a time horizon $T > 0$.  We equip the space $D([0,T], \mathcal{D}^\prime (\T^d))$, the space of Càdlàg $\mathcal{D}^\prime (\T^d)$-valued trajectories on $[0,T]$, with the uniform weak topology: a sequence $\{\mu^n_\cdot\}_{n \geq 1}$  converges to $\mu_\cdot$ in $D([0,T], \mathcal{D}^\prime (\T^d))$ if and only if for any $H \in \mathcal{D} (\T^d)$, 
\[\lim_{n \rightarrow \infty} \sup_{0 \leq t \leq T} |\<\mu^n_t,H\> - \<\mu_t,H\>| = 0.\]

Assume the initial measure of the process is $\nu^n_{\rho_*}$. We are interested in the rescaled density fluctuation field $\mu^n_t \in \mathcal{D}^\prime (\T^d)$ of the process, which acts on $H \in \mathcal{D} (\T^d)$ as
\[\<\mu^n_t,H\> = \frac{1}{a_n} \sum_{x \in \T_n^d} \bar{\eta}_x (t) H(\tfrac{x}{n}),\]
where $\bar{\eta}_x = \eta_x - \rho_*$, and $\{a_n\}_{n \geq }$ is a sequence of positive real numbers.  Define
\[g_d (n) = \begin{cases}
	n, \quad &\text{if } d=1,\\
	\log	n, \quad &\text{if } d=2,\\
	1, \quad &\text{if } d \geq 3.
\end{cases}\]
We shall need the following assumptions on $a_n$.

\begin{assumption}\label{assump: a n}
\[n^{d-1} \sqrt{g_d (n)} \ll a_n \ll n^d.\]
\end{assumption}

\begin{remark}
	In the regime of moderate deviations, one usually assumes $n^{d/2} \ll a_n \ll n^d$. So, the above assumption is only optimal in dimension one. 
\end{remark}

Next, we introduce the MDP rate function.  For $\mu_0 \in \mathcal{D}^\prime (\T^d)$, the rate function corresponding to the initial distribution is defined as
\begin{equation}\label{q 0}
\mathcal{Q}_0 (\mu_0) := \sup_{\phi \in \mathcal{D} (\T^d)} \Big\{ \<\mu_0,\phi\> - \frac{\chi(\rho_*)}{2} \|\phi\|_{L^2 (\T^d)}^2 \Big\},
\end{equation}
where $\chi (\rho_*) = \rho_* (1-\rho_*)$. Let $C^{1,\infty} ([0,T] \times \T^d)$ be the space of functions on $[0,T] \times \T^d$ which are continuously derivative in the time variable and smooth in the space variable. For $\mu \in D([0,T], \mathcal{D}^\prime (\T^d))$ and $H \in C^{1,\infty} ([0,T] \times \T^d)$, define
\[\ell_T (\mu,H) := \<\mu_T,H_T\> - \<\mu_0,H_0\> - \int_0^T \<\mu_s, (\partial_s+\Delta +F^\prime (\rho_*)) H_s\> ds.\]
By simple calculations, $F^\prime (\rho_*) = \lambda -a - b - 2\lambda \rho_*$.
The rate function corresponding to the dynamics of the process  is defined as
\[\mathcal{Q}_{\rm dyn} (\mu) := \sup_{H \in C^{1,\infty} ([0,T] \times \T^d)} \Big\{ \ell_T (\mu,H)  -  \chi (\rho_*) \sum_{i=1}^d \|\partial_{u_i} H\|_{L^2 ([0,T] \times \T^d)}^2 - \frac{G(\rho_*)}{2} \|H\|_{L^2 ([0,T] \times \T^d)}^2\Big\},\]
where $G(\rho) := E_{\nu^n_\rho} [c_x (\eta) ] = (a-\lambda \rho) (1-\rho) + b \rho$. Finally, the rate function is defined as \[\mathcal{Q}_T := \mathcal{Q}_0 + \mathcal{Q}_{\rm dyn}.\]

Denote by $\P^n_{\rho_*}$ the probability measure on $D([0,T],\Omega_n)$ induced by the process $(\eta(t))_{t \geq 0}$ with initial distribution $\nu^n_{\rho_*}$, and by $\E^n_{\rho_*}$ the corresponding expectation.

Below is the main result of the article.

\begin{theorem}\label{thm mdp}
Assume Assumption \ref{assump: a n} holds.  There exists $\lambda_c > 0$ such that for any $\lambda \in [-\lambda_c,\lambda_c]$, under $\P^n_{\rho_*}$, the sequence of measures $\{\mu^n_t, 0 \leq t \leq T\}_{n \geq 1}$ satisfies the moderate deviation principles with decay rate $a_n^2 / n^d$ and with rate function $\mathcal{Q}_T$. Precisely speaking, for any closed set $C$ and any open set $O$ in $D([0,T], \mathcal{D}^\prime (\T^d))$,
\begin{align*}
	\limsup_{n \rightarrow \infty} \frac{n^d}{a_n^2} \log \P^n_{\rho_*} \Big( \{\mu^n_t, 0 \leq t \leq T\}_{n \geq 1} \in C\Big) \leq - \inf_{\mu \in C} \mathcal{Q}_T (\mu),\\
		\liminf_{n \rightarrow \infty} \frac{n^d}{a_n^2} \log \P^n_{\rho_*} \Big( \{\mu^n_t, 0 \leq t \leq T\}_{n \geq 1} \in O\Big) \geq - \inf_{\mu \in O} \mathcal{Q}_T (\mu).
\end{align*}
\end{theorem}

\section{The upper bound}\label{sec: upper bound}

In Subsection \ref{subsec:exp martingale}, we introduce an exponential martingale related to the process and express it as a functional of the fluctuation field $\mu^n_t$ plus super-exponentially small error terms.   This is the main step of the upper bound. The main difficulty is to prove two super-exponential estimates, see Lemmas \ref{lem: super-esponential estimate 1} and \ref{lem: super-esponential estimate 2}.  We prove exponential tightness of the rescaled density fluctuation field in Subsection \ref{subsec exp tight}. Finally, the proof of the upper bound is concluded in Subsection \ref{subsec pf upper}.

\subsection{An exponential martingale}\label{subsec:exp martingale} For any $H \in C^{1,\infty} ([0,T] \times \T^d)$, by Feynman-Kac formula, 
\begin{equation*}
	\mathcal{M}^n_t (H) = \exp \Big\{  \frac{a_n^2}{n^d}  \<\mu^n_t,H_t\> - \frac{a_n^2}{n^d} \<\mu^n_0,H_0\> - \int_0^t e^{-\tfrac{a_n^2}{n^d} \langle\mu^n_s,H_s\rangle} \big(\partial_s + L_n\big) e^{\tfrac{a_n^2}{n^d} \langle\mu^n_s,H_s\rangle} ds \Big\}
\end{equation*} 
is an exponential martingale.  By direct calculations,
\begin{equation*}
	e^{-\tfrac{a_n^2}{n^d} \langle\mu^n_s,H_s\rangle} \partial_s  e^{\tfrac{a_n^2}{n^d} \langle\mu^n_s,H_s\rangle} = \frac{a_n^2}{n^d} \<\mu^n_s,\partial_s H_s\>,
\end{equation*}
and
\begin{align*}
	e^{-\tfrac{a_n^2}{n^d} \langle\mu^n_s,H_s\rangle} L_n  e^{\tfrac{a_n^2}{n^d} \langle\mu^n_s,H_s\rangle}  &=  n^2 \sum_{x \in \T_n^d} \sum_{i=1}^d \big[\exp \big\{ \tfrac{a_n}{n^d} \big(\eta_{x} (s) - \eta_{x+e_i} (s)\big) \big(H_s(\tfrac{x+e_i}{n}) - H_s (\tfrac{x}{n})\big)\big\} - 1\big] \\
	&+ \sum_{x \in \T_n^d}  c_x (\eta(s)) \big[\exp \big\{ \tfrac{a_n}{n^d} (1-2\eta_x(s)) H_s (\tfrac{x}{n})\big\} - 1 \big].
\end{align*}
By Taylor's expansion up to the second order, and using the  summation by parts formula,
\begin{multline*}
	e^{-\tfrac{a_n^2}{n^d} \langle\mu^n_s,H_s\rangle} L_n  e^{\tfrac{a_n^2}{n^d} \langle\mu^n_s,H_s\rangle}  =  \frac{a_n^2}{n^d} \Big\{ \<\mu^n_s,\Delta_n H_s\> + \frac{1}{2n^d} \sum_{x \in \T_n^d}  \sum_{i=1}^d \big(\eta_x (s) - \eta_{x+e_i} (s)\big)^2 \big( \nabla_{n,i} H_s \big)^2 (\tfrac{x}{n}) \\ + \frac{1}{a_n} \sum_{x \in \T_n^d} c_x (\eta(s)) (1-2\eta_x (s)) H_s (\tfrac{x}{n}) + \frac{1}{2n^d} \sum_{x \in \T_n^d} c_x (\eta(s)) H_s^2  (\tfrac{x}{n})+ R^{n,1} _s (H)
	 \Big\},
\end{multline*}
where $\Delta_n = \sum_{i=1}^{d} \Delta_{n,i}$ is the discrete Laplacian,  and
\[\Delta_{n,i} H_s (\tfrac{x}{n}) = n^2 \big( H_s (\tfrac{x+e_i}{n}) + H_s (\tfrac{x-e_i}{n}) - 2 H_s (\tfrac{x}{n})  \big), \quad  \nabla_{n,i} H_s (\tfrac{x}{n}) = n \big( H_s (\tfrac{x+e_i}{n}) -  H_s (\tfrac{x}{n})\big).\]
The error term satisfies \[|R^{n,1}_s (H)| \leq C(H) a_n / n^d.\] Since $F(\rho_\ast) = E_{\nu^n_{\rho_\ast}} \big[c_x (\eta) (1-2\eta_x)\big] = 0$, we have
\begin{align*}
c_x (\eta) (1-2\eta_x) &= \Big(a+\frac{\lambda}{2d} \sum_{y \sim x} \eta_y\Big) (1-\eta_x) - b \eta_x \\&= -(a+b+\lambda \rho_*) \bar{\eta}_x + \frac{\lambda (1-\rho_*)}{2d}  \sum_{y \sim x} \bar{\eta}_y 
- \frac{\lambda}{2d} \sum_{y \sim x} \bar{\eta}_y \bar{\eta}_x. 
\end{align*}
Thus, using the summation by parts formula, 
\begin{equation}\label{eqn 6}
\begin{aligned}
	&e^{-\tfrac{a_n^2}{n^d} \langle\mu^n_s,H_s\rangle} L_n  e^{\tfrac{a_n^2}{n^d} \langle\mu^n_s,H_s\rangle}  \\
	=&  \frac{a_n^2}{n^d} \Big\{ \<\mu^n_s,\big((1+\tfrac{\lambda(1-\rho_*)}{2dn^2})\Delta_n + F^\prime (\rho_*) \big) H_s\> 
	+ \frac{1}{2n^d} \sum_{x \in \T_n^d}  \sum_{i=1}^d \big(\eta_x (s) - \eta_{x+e_i} (s)\big)^2 \big( \nabla_{n,i} H_s \big)^2 (\tfrac{x}{n})  \\
	&- \frac{\lambda}{2da_n} \sum_{x \in \T_n^d} \sum_{y \sim x}  \bar{\eta}_{x} (s)\bar{\eta}_{y} (s) H_s (\tfrac{x}{n}) 
	+ \frac{1}{2n^d} \sum_{x \in \T_n^d} c_x (\eta(s)) H_s^2  (\tfrac{x}{n})+R^{n,1} _s (H)
	\Big\}.
\end{aligned}
\end{equation}

In order to further deal with the above terms,  we need the following two super-exponential estimates in Lemmas \ref{lem: super-esponential estimate 1} and \ref{lem: super-esponential estimate 2}.

\begin{lemma}[A super-exponential estimate for degree-one terms]\label{lem: super-esponential estimate 1}
	If $r_n \gg a_n$, then for any $H \in \mathcal{D} (\T^d)$ and for any $\varepsilon > 0$,
	\begin{equation}\label{superexp estimate 1}
	\lim_{n \rightarrow \infty}	\frac{n^d}{a_n^2} \log \P^n_{\rho_*} \Big( \sup_{0 \leq t \leq T} \Big| \int_0^t \frac{1}{r_n} \sum_{x \in \T_n^d} \bar{\eta}_x (s) H_s (\tfrac{x}{n}) ds\Big| > \varepsilon  \Big) = -\infty.
	\end{equation}
\end{lemma}

Before proving the above lemma, we first recall the following version of Garsia-Rodemich-Rumsey inequality (see \cite[page 182]{klscaling}), which allows us to remove the supremum over time inside the above inequality.

\begin{lemma}[Garsia-Rodemich-Rumsey inequality]\label{lem: GRR inequality}
	Let $g: [0,T] \rightarrow \R$ be a continuous function. Assume $\psi (u), p (u)$ are strictly increasing functions satisfying that  
	\[\psi (0) = p (0) = 0, \quad \lim_{u \rightarrow \infty} \psi (u) = \infty.\]
	Then, for any $\delta > 0$,
	\[\sup_{|t-s| < \delta \atop 0 \leq s,t \leq T} |g(t) - g(s)| \leq 8 \int_0^\delta \psi^{-1} \Big(\frac{4B}{u^2} \Big) p (du),\]
	where
	\[B = \int_0^T ds \int_0^T dt \psi \Big( \frac{|g(t) - g(s)|}{p (|t-s|)}\Big).\]
\end{lemma}

\begin{proof}[Proof of Lemma \ref{lem: super-esponential estimate 1}] First, we adopt the ideas in \cite{gao2003moderate}, where  the Garsia-Rodemich-Rumsey inequality was used to prove exponential tightness of the density fluctuation field, to remove the supremum over time inside the probability in \eqref{superexp estimate 1}.  In Lemma \ref{lem: GRR inequality}, take
\[g(t) =  \int_0^t \frac{1}{r_n} \sum_{x \in \T_n^d} \bar{\eta}_x (s) H_s (\tfrac{x}{n}) ds,\quad  \psi (u) = u^{2q}, \quad p(u) = u^{1/3},\]
where $q > 3$ is fixed. Then, there exists some constant $C_1 = C_1 (q)$ such that
\[\sup_{0 \leq t \leq T} \Big| \int_0^t \frac{1}{r_n} \sum_{x \in \T_n^d} \bar{\eta}_x (s) H_s (\tfrac{x}{n}) ds\Big| \leq C_1 \int_0^T \Big( \frac{B}{u^2}\Big)^{\tfrac{1}{2q}} \frac{du}{u^{2/3}}  = C_1 T^{\tfrac{1}{3} - \tfrac{1}{q}} B^{\tfrac{1}{2q}} =: C_2 (q,T) B^{\tfrac{1}{2q}},\]
where 
\[B = \int_0^T dt \int_0^T dt^\prime \Big( \frac{1}{|t-t^\prime|^{1/3}} \big| \int_t^{t^\prime} \frac{1}{r_n} \sum_{x \in \T_n^d} \bar{\eta}_x (s) H_s (\tfrac{x}{n}) ds \big| \Big)^{2q}.\]
Thus, by Markov's inequality, for any $A > 0$,
\begin{align*}
&\frac{n^d}{a_n^2} \log 	 \P^n_{\rho_*} \Big( \sup_{0 \leq t \leq T} \Big| \int_0^t \frac{1}{r_n} \sum_{x \in \T_n^d} \bar{\eta}_x (s) H_s (\tfrac{x}{n}) ds\Big| > \varepsilon  \Big)  \\
\leq& \frac{n^d}{a_n^2} \log \P^n_{\rho_*} \Big( B^{\tfrac{1}{2q}} > \frac{\varepsilon}{C_2}\Big) \\
\leq& - \frac{A \varepsilon}{C_2 T^{1/q}} + \frac{n^d}{a_n^2} \log \E^n_{\rho_*} \Big[ \exp \Big\{  \frac{a_n^2A}{n^d} \Big(\frac{B}{T^2}\Big)^{\tfrac{1}{2q}} \Big\}\Big].
\end{align*}
Since $A$ could be taken arbitrarily large, we only need to show that,  for any $A > 0$,
\begin{equation*}
	\lim_{n \rightarrow \infty} \frac{n^d}{a_n^2} \log \E^n_{\rho_*} \Big[ \exp \Big\{  \frac{a_n^2 A}{n^d} \Big(\frac{B}{T^2}\Big)^{\tfrac{1}{2q}} \Big\}\Big] = 0.
\end{equation*}
Define 
\[f_q (x) = \exp \Big\{   \big((2q-1)^{2q} +x \big)^{\tfrac{1}{2q}}\Big\}, \quad x \geq 0.\]	
One could check directly that $f_q$ is convex.  Then,
\begin{align*}
	\exp \Big\{  \frac{a_n^2 A}{n^d} \Big(\frac{B}{T^2}\Big)^{\tfrac{1}{2q}} \Big\} &\leq f_q \Big( \frac{1}{T^2}  \int_0^T dt \int_0^T dt^\prime  \Big( \frac{a_n^2 A}{n^d |t-t^\prime|^{1/3}} \big| \int_t^{t^\prime} \frac{1}{r_n} \sum_{x \in \T_n^d} \bar{\eta}_x (s) H_s (\tfrac{x}{n}) ds \big| \Big)^{2q} \Big) \\
	&\leq \frac{1}{T^2}  \int_0^T dt \int_0^T dt^\prime f_q \Big( \Big( \frac{a_n^2 A}{n^d|t-t^\prime|^{1/3} } \big| \int_t^{t^\prime} \frac{1}{r_n} \sum_{x \in \T_n^d} \bar{\eta}_x (s) H_s (\tfrac{x}{n}) ds \big| \Big)^{2q} \Big)\\
	& \leq e^{2q-1} \frac{1}{T^2}  \int_0^T dt \int_0^T dt^\prime  \exp \Big\{ \frac{a_n^2 A}{n^d |t-t^\prime|^{1/3}} \big| \int_t^{t^\prime} \frac{1}{r_n} \sum_{x \in \T_n^d} \bar{\eta}_x (s) H_s (\tfrac{x}{n}) ds \big|  \Big\},
\end{align*}
where we used the inequality $f_q (x) \leq \exp \{(2q-1) + x^{\tfrac{1}{2q}}\}$. Thus, we only need to show that, for any $A > 0$,
\begin{equation}\label{eqn 2}
	\lim_{n \rightarrow \infty} \frac{n^d}{a_n^2} \log \Big(  \frac{1}{T^2}  \int_0^T dt \int_0^T dt^\prime \,\E^n_{\rho_*} \Big[  \exp \Big\{ \frac{a_n^2 A}{n^d |t-t^\prime|^{1/3} } \big| \int_t^{t^\prime} \frac{1}{r_n} \sum_{x \in \T_n^d} \bar{\eta}_x (s) H_s (\tfrac{x}{n}) ds \big|  \Big\}\Big]\Big) = 0.
\end{equation}

Next, we bound the expectation
\[\E^n_{\rho_*} \Big[  \exp \Big\{ \frac{a_n^2 A}{n^d |t-t^\prime|^{1/3}} \big| \int_t^{t^\prime} \frac{1}{r_n} \sum_{x \in \T_n^d} \bar{\eta}_x (s) H_s (\tfrac{x}{n}) ds \big|  \Big\}\Big].\]
Without loss of generality, assume  $t < t^\prime$. Using the basic inequalities $e^{|x|} \leq e^x + e^{-x}$ and $\log (x+y) \leq \log 2 + \max \{\log x, \log y\}$, we can remove the absolute value inside the above exponential. Then, we shall use the following version of  Feynman-Kac formula, see \cite[Proposition 4.3]{gonccalves2024clt} for example. Below, we say $f$ is a $\nu^n_{\rho_*}$-density if $f \geq 0$ and $E_{\nu^n_{\rho_*}} [f] = 1$.

\begin{lemma}\label{lem: feynman-kac formula}
Let $W$ be a function on $[0,T] \times \Omega_n$.  Then,
\[\log \E^n_{\rho_*} \Big[\exp \Big\{  \int_0^T W(t,\eta(t)) dt \Big\}  \Big] \leq \int_0^T dt \sup_{f: \nu^n_{\rho_*}\text{-density}} \Big\{  \int \Big(W_t + \frac{1}{2} L_n^* \mathbf{1} \Big) f d \nu^n_{\rho_*} - \int \Gamma_n (\sqrt{f}) d \nu^n_{\rho_*}\Big\},\]
where $L_n^*$ is the adjoint operator of $L_n$ in $L^2 (\nu^n_{\rho_*})$, and $\Gamma_n  (f) = \Gamma_n^{\rm ex} (f) + \Gamma_n^{\rm r} (f)$, where
\begin{align*}
\Gamma_n^{\rm ex} (f)  &= \frac{n^2}{2} \sum_{x \in \T_n^d} \sum_{i=1}^d \big[f(\eta^{x,x+e_i}) - f(\eta)\big]^2,\\
\Gamma_n^{\rm r} (f)  &= \frac{1}{2} \sum_{x \in \T_n^d} c_x (\eta) \big[f(\eta^{x}) - f(\eta)\big]^2.
\end{align*}
\end{lemma}

\medspace

Take
\[W(s,\eta) = \frac{a_n^2 A}{n^dr_n|t-t^\prime|^{1/3} }  \sum_{x \in \T_n^d} \bar{\eta}_x H_s (\tfrac{x}{n}), \quad s \in [t,t^\prime],\]
and $W(s,\eta) = 0$ otherwise. Then, we bound
\begin{equation}\label{eqn 1}
\log \E^n_{\rho_*} \Big[  \exp \Big\{ \frac{a_n^2 A}{n^d|t-t^\prime|^{1/3} } \big| \int_t^{t^\prime} \frac{1}{r_n} \sum_{x \in \T_n^d} \bar{\eta}_x (s) H_s (\tfrac{x}{n}) ds \big|  \Big\}\Big] \leq \int_0^{t^\prime} \Lambda^n_s ds,
\end{equation}
where, for $s < t$,
\[\Lambda^n_s = \sup_{f: \nu^n_{\rho_*}\text{-density}} \Big\{  \int  \frac{1}{2} L_n^* \mathbf{1}  f d \nu^n_{\rho_*} - \int \Gamma_n (\sqrt{f}) d \nu^n_{\rho_*}\Big\},\]
and for $t \leq s \leq t^\prime$,
\[\Lambda^n_s = \sup_{f: \nu^n_{\rho_*}\text{-density}} \Big\{  \int \Big( \frac{a_n^2 A}{n^dr_n|t-t^\prime|^{1/3}}  \sum_{x \in \T_n^d} \bar{\eta}_x H_s (\tfrac{x}{n}) + \frac{1}{2} L_n^* \mathbf{1}   \Big) f d \nu^n_{\rho_*} - \int \Gamma_n (\sqrt{f}) d \nu^n_{\rho_*}\Big\}.\]

To bound  $\Lambda^n_s$, we first recall the so-called main lemma, which was proved in \cite[Theorem 3.3]{gonccalves2024clt}.  For $1 \leq i \leq d$, let $h^i$ be functions on $\T_n^d$.  Define
\[V(h) := \sum_{x \in \T_n^d} \sum_{i=1}^d \bar{\eta}_x \bar{\eta}_{x+e_i} h^i (x).\]

\begin{proposition}[The first main lemma]\label{prop: main lemma}
There exists some constant $C_0 =C_0 (d)$ such that for any $\nu^n_{\rho_*}$-density $f$,
\[	\int V(h) f d \nu^n_{\rho_*} \leq \frac{1}{4} \int \Gamma_n^{\rm ex} (\sqrt{f}) d \nu^n_{\rho_*} + C_0 \kappa (\rho_*) \mathcal{A} (\|h\|_\infty) \int \Gamma_n^{\rm r} (\sqrt{f}) d \nu^n_{\rho_*}  + C_0  \mathcal{A} (\|h\|_\infty) n^{d-2} g_d (n),\]
where $\mathcal{A} (u) := u (1+u)$, $\|h\|_{\infty} := \max_{x \in \T_n^d, 1 \leq i \leq d} |h^i (x)|$, and
\[ \kappa (\rho) = \frac{2 \rho (1-\rho) | \log \tfrac{\rho}{1-\rho}|}{
	\min \{a,a+\lambda,b\} |1-2\rho|}.\]
\end{proposition}

To calculate $L_n^* \mathbf{1}$, we use the following formula, see \cite[Proposition 3.3]{jaram20nonequilireaction} for example.

\begin{lemma}\label{lem 2}
	Let $X_t$ be a continuous markov process on some space $\Omega$ with generator 
	\[L f (x) = \sum_{y \in \Omega} r(x,y) [f(y) - f(x)].\]
	Let $\nu$ be a probability measure on $\Omega$ and $L^*$ be the adjoint of $L$ in $L^2 (\nu)$. Then,
	\[L^* \mathbf{1} (x) = \sum_{y \neq x} \Big\{  \frac{\nu(y) r(y,x)}{\nu(x)} - r(x,y) \Big\}.\]
\end{lemma}

We first bound the term $\Lambda^n_s$ for $s < t$. Using the above lemma,
\begin{align*}
L_n^* \mathbf{1} &= \sum_{x \in \T_n^d}  \Big\{  \eta_x \Big(  \frac{1-\rho_*}{\rho_*}  \Big[a+ \frac{\lambda}{2d} \sum_{y \sim x} \eta_y\Big] -b\Big) + (1-\eta_x) \Big(  \frac{\rho_*}{1-\rho_*}   b  -  \Big[a+ \frac{\lambda}{2d} \sum_{y \sim x} \eta_y\Big]  \Big) \Big\}\\
&= \sum_{x \in \T_n^d}  \Big(\frac{\eta_x}{\rho_*} - \frac{1-\eta_x}{1-\rho_*}\Big)  \Big(\Big[a+ \frac{\lambda}{2d} \sum_{y \sim x} \eta_y\Big] (1-\rho_*) - b \rho_*\Big)\\
&= \frac{1}{\chi (\rho_*)}\sum_{x \in \T_n^d} \bar{\eta}_x \Big( F(\rho_*) + \frac{\lambda (1-\rho_*)}{2d} \sum_{y \sim x} \bar{\eta}_y \Big)\\
&= \frac{\lambda}{2d \rho_*} \sum_{x \in \T_n^d} \sum_{y \sim x} \bar{\eta}_x \bar{\eta}_{y}.
\end{align*}
Thus, by Proposition \ref{prop: main lemma}, for $s < t$,
\begin{align*}
\Lambda^n_s \leq \sup_{f: \nu^n_{\rho_*}\text{-density}} \Big\{   \Big[ C_0 \kappa (\rho_*) \mathcal{A} \Big(\frac{\lambda}{2d \rho_*}\Big) - 1 \Big] \int \Gamma_n^{\rm r} (\sqrt{f}) d \nu^n_{\rho_*}  + C_0  \mathcal{A}  \Big(\frac{\lambda}{2d \rho_*}\Big) n^{d-2} g_d (n) \Big\}.
\end{align*}
Let $\lambda_c > 0$ be such that
\begin{equation}\label{lambda condition}
C_0 \kappa (\rho_*)  \mathcal{A} \Big(\frac{\lambda}{2d \rho_*}\Big ) < 1 
\end{equation}
for any $\lambda \in [- \lambda_c,\lambda_c]$, which is possible since 
\[\lim_{\lambda \rightarrow 0}C_0 \kappa (\rho_*)  \mathcal{A} \Big(\frac{\lambda}{2d \rho_*}\Big ) = 0. \]
Then, we bound, for $s < t$, 
\[\Lambda^n_s \leq C_0  \mathcal{A}  \Big(\frac{\lambda}{2d \rho_*}\Big) n^{d-2} g_d (n).\]
Next, we deal with the case $t \leq s \leq t^\prime$. We need  to consider the extra term  
	\[\int \frac{a_n^2 A}{n^dr_n |t-t^\prime|^{1/3}}  \sum_{x \in \T_n^d} \bar{\eta}_x H_s (\tfrac{x}{n})  f(\eta)d\nu^n_{\rho_*}.\]
By making the change of transformations $\eta \mapsto \eta^x$, the last term equals
	\begin{align*}
	&\int \frac{a_n^2 A}{n^dr_n|t-t^\prime|^{1/3} }  \sum_{x \in \T_n^d} (1-\eta_x - \rho_*)  H_s (\tfrac{x}{n}) f(\eta^x) \frac{\nu^n_{\rho_*} (\eta^x)}{\nu^n_{\rho_*} (\eta)} d\nu^n_{\rho_*}\\
	=&- \int \frac{a_n^2 A}{n^dr_n|t-t^\prime|^{1/3} } \sum_{x \in \T_n^d} \bar{\eta}_x  H_s (\tfrac{x}{n}) f(\eta^x) d\nu^n_{\rho_*}\\
	=& \frac{1}{2}\int \frac{a_n^2 A}{n^dr_n|t-t^\prime|^{1/3} } \sum_{x \in \T_n^d} \bar{\eta}_x  H_s (\tfrac{x}{n}) [f(\eta) - f(\eta^x) ]d\nu^n_{\rho_*}.
	\end{align*}
	By Young's inequality, for any $\varepsilon_0 > 0$,  the last line is bounded by
	\begin{align*}
		&\frac{\varepsilon_0}{2} \int \sum_{x \in \T_n^d} c_x (\eta) [\sqrt{f(\eta)} - \sqrt{f(\eta^x)} ]^2d\nu^n_{\rho_*} \\
		&\qquad + \frac{A^2 a_n^4}{8 \varepsilon_0 n^{2d} r_n^2 |t-t^\prime|^{2/3}} \int \sum_{x \in \T_n^d} c_x (\eta)^{-1} \big(\bar{\eta}_x\big)^2  H^2_s (\tfrac{x}{n})  [\sqrt{f(\eta)} + \sqrt{f(\eta^x)} ]^2d\nu^n_{\rho_*}\\
		&\leq \varepsilon_0 \int \Gamma_n^{\rm r} \big(\sqrt{f}\big) d \nu^n_{\rho_*} + \frac{C_3A^2a_n^4}{\varepsilon_0 n^d r_n^2 |t-t^\prime|^{2/3}} 
	\end{align*}
for some constant $C_3 = C_3 (\rho_*,a,b,\lambda,H)$.  Thus, for $t \leq s \leq t^\prime$,  for any $\varepsilon_0 > 0$,
\begin{align*}
	\Lambda^n_s \leq \sup_{f: \nu^n_{\rho_*}\text{-density}} \Big\{   &\Big[ C_0 \kappa (\rho_*) \mathcal{A} \Big(\frac{\lambda}{2d \rho_*}\Big) 
	+ \varepsilon_0 - 1 \Big] \int \Gamma_n^{\rm r} (\sqrt{f}) d \nu^n_{\rho_*}  \\
	&+ C_0  \mathcal{A}  \Big(\frac{\lambda}{2d \rho_*}\Big) n^{d-2} g_d (n) 
	 +  \frac{C_3A^2a_n^4}{\varepsilon_0 n^d r_n^2 |t-t^\prime|^{2/3}} \Big\}.
\end{align*}
By choosing $\varepsilon_0$ small enough such that $C_0 \kappa (\rho_*) \mathcal{A} (\frac{\lambda}{2d \rho_*}) + \varepsilon_0 < 1$, we have that, for $t \leq s \leq t^\prime$,
\[\Lambda^n_s \leq C_0  \mathcal{A}  \Big(\frac{\lambda}{2d \rho_*}\Big) n^{d-2} g_d (n) +  \frac{C_3A^2a_n^4}{\varepsilon_0 n^d r_n^2 |t-t^\prime|^{2/3}}.\]

Recall \eqref{eqn 1}. Then,
\begin{align*}
	&\log \E^n_{\rho_*} \Big[  \exp \Big\{ \frac{a_n^2 A}{n^d\sqrt{|t-t^\prime|}} \big| \int_t^{t^\prime} \frac{1}{r_n} \sum_{x \in \T_n^d} \bar{\eta}_x (s) H_s (\tfrac{x}{n}) ds \big|  \Big\}\Big] \\ &\leq \int_0^t C_0  \mathcal{A}  \Big(\frac{\lambda}{2d \rho_*}\Big) n^{d-2} g_d (n) ds 
	 + \int_t^{t^\prime}  \Big[C_0  \mathcal{A}  \Big(\frac{\lambda}{2d \rho_*}\Big) n^{d-2} g_d (n) +  \frac{C_3A^2a_n^4}{\varepsilon_0 n^d r_n^2 |t-t^\prime|^{2/3}} \Big] ds\\
	 &\leq C_0  t^\prime \mathcal{A}  \Big(\frac{\lambda}{2d \rho_*}\Big) n^{d-2} g_d (n) + \frac{C_3A^2a_n^4}{\varepsilon_0 n^d r_n^2 } |t-t^\prime|^{1/3}.
\end{align*}
Finally, we bound the limit in \eqref{eqn 2} by
\begin{equation}\label{eqn 7}
\begin{aligned}
	&\lim_{n \rightarrow \infty} \frac{n^d}{a_n^2} \log \Big(  \frac{1}{T^2}  \int_0^T dt \int_0^T dt^\prime \, \exp \Big\{  C_0  t^\prime \mathcal{A}  \Big(\frac{\lambda}{2d \rho_*}\Big) n^{d-2} g_d (n) + \frac{C_3A^2a_n^4}{\varepsilon_0 n^d r_n^2 } |t-t^\prime|^{1/3} \Big\}\Big)\\
	\leq& \lim_{n \rightarrow \infty} \Big\{ C_0 T  \mathcal{A}  \Big(\frac{\lambda}{2d \rho_*}\Big)  \frac{n^{2d-2} g_d (n)}{a_n^2} + \frac{C_3 (2T)^{1/3} A^2a_n^2}{\varepsilon_0 r_n^2} \Big\}.
\end{aligned}
\end{equation}
We conclude the proof by using the assumption $n^{d-1} \sqrt{g_d (n)} \ll a_n \ll r_n$.
\end{proof}

\begin{lemma}[A super-exponential estimate for degree-two terms]\label{lem: super-esponential estimate 2}
For any $H \in \mathcal{D} (\T^d)$, for any $1 \leq i \leq d$,  and for any $\varepsilon > 0$,
	\begin{equation}
		\lim_{n \rightarrow \infty}	\frac{n^d}{a_n^2} \log \P^n_{{\rho_\ast}} \Big( \sup_{0 \leq t \leq T} \Big|\int_0^t \frac{1}{a_n} \sum_{x \in \T_n^d} \bar{\eta}_x (s) \bar{\eta}_{x+e_i} (s) H_s (\tfrac{x}{n}) ds\Big| > \varepsilon  \Big) = -\infty.
	\end{equation}
\end{lemma}

\begin{proof}
Following the proof of \eqref{eqn 2} line by line, we only need to show that, for any $A > 0$,
	\begin{multline}\label{eqn 3}
		\lim_{n \rightarrow \infty} \frac{n^d}{a_n^2} \log \Big(  \frac{1}{T^2}  \int_0^T dt \int_0^T dt^\prime \\
		\E^n_{\rho_*} \Big[  \exp \Big\{ \frac{a_n A}{n^d |t-t^\prime|^{1/3}} \Big| \int_t^{t^\prime}  \sum_{x \in \T_n^d} \bar{\eta}_x (s) \bar{\eta}_{x+e_i} (s) H_s (\tfrac{x}{n}) ds \Big|  \Big\}\Big]\Big) = 0.
	\end{multline}
	As before, we assume $t < t^\prime$. By Lemma \ref{lem: feynman-kac formula}, the expectation in \eqref{eqn 3} is bounded by $\exp \{ \int_0^{t^\prime} \Lambda^n_s ds\}$, where for $t \leq s \leq t^\prime$,
	\begin{equation}
	 \Lambda^n_s = \sup_{f: \nu^n_{\rho_*}\text{-density}} \Big\{ \int \Big( \frac{a_n A}{n^d |t-t^\prime|^{1/3}} \sum_{x \in \T_n^d} \bar{\eta}_x  \bar{\eta}_{x+e_i} H_s (\tfrac{x}{n}) + \frac{1}{2} L_n^* \mathbf{1}  \Big) f (\eta)d\nu^n_{\rho_*} - \int \Gamma_n (\sqrt{f}) d \nu^n_{\rho_*}\Big\}.
	\end{equation}
	and for $s < t$, 
	\[\Lambda^n_s = \sup_{f: \nu^n_{\rho_*}\text{-density}} \Big\{ \int  \frac{1}{2} L_n^* \mathbf{1} f (\eta)d\nu^n_{\rho_*} - \int \Gamma_n (\sqrt{f}) d \nu^n_{\rho_*}\Big\}.\]
	We have shown in the proof of Lemma \ref{lem: super-esponential estimate 1} that, for $s < t$,
	\[\Lambda^n_s \leq C_0  \mathcal{A}  \Big(\frac{\lambda}{2d \rho_*}\Big) n^{d-2} g_d (n).\]
	In particular,
	\begin{equation}\label{eqn 4}
	\int_0^t \Lambda^n_s ds  \leq C_0 T  \mathcal{A}  \Big(\frac{\lambda}{2d \rho_*}\Big) n^{d-2} g_d (n).
	\end{equation}
	For $t \leq s \leq t^\prime$, by Proposition \ref{prop: main lemma},
	\begin{align*}
		\Lambda^n_s \leq \sup_{f: \nu^n_{\rho_*}\text{-density}} \Big\{   &\Big[ C_0 \kappa (\rho_*) \mathcal{A} \Big(\frac{\lambda}{2d \rho_*} + \frac{a_n A}{n^d |t-t^\prime|^{1/3}} \|H\|_\infty \Big) 
		 - 1 \Big] \int \Gamma_n^{\rm r} (\sqrt{f}) d \nu^n_{\rho_*}  \\
		&+ C_0  \mathcal{A}  \Big(\frac{\lambda}{2d \rho_*}+  \frac{a_n A}{n^d |t-t^\prime|^{1/3}} \|H\|_\infty \Big) n^{d-2} g_d (n) \Big\}.
	\end{align*}
	Note that we cannot deal with the coefficient before $\Gamma_n^{\rm r} (\sqrt{f})$ by using \eqref{lambda condition} directly since there is a singularity at the point $t=t^\prime$.  The ideas is as follows. Fix $\varepsilon_0 > 0$ small enough such that $C_0 \kappa (\rho_*) \mathcal{A} (\tfrac{\lambda}{2d\rho_*} + \varepsilon_0) < 1$.   One one hand, if \[\frac{a_n A}{n^d |t-t^\prime|^{1/3}} \|H\|_\infty  < \varepsilon_0,\] then 
	\[\int_t^{t^\prime} \Lambda^n_s \leq C_4 T n^{d-2} g_d (n)\]
	for some constant $C_4 = C_4 (C_0,\lambda,d,\rho_*,H,A)$. On the other hand, if 
	\[\frac{a_n A}{n^d |t-t^\prime|^{1/3}} \|H\|_\infty  \geq  \varepsilon_0,\]
	then, since $\mathcal{A} (x+y) = \mathcal{A} (x) + (1+2x) y + y^2$,
	\begin{equation}\label{eqn 5}
		\begin{aligned}
		&\int_t^{t^\prime} \Lambda^n_s ds \\
		&\leq \sup_{f} 
		\Big\{    C_0 \kappa (\rho_*) \Big( \big(1+\tfrac{\lambda}{d\rho_*}\big) \frac{a_n A |t-t^\prime|^{2/3}}{n^d} \|H\|_\infty + \frac{a_n^2 A^2 |t-t^\prime|^{1/3}}{n^{2d}} \|H\|_\infty^2\Big)  \int \Gamma_n^{\rm r} (\sqrt{f}) d \nu^n_{\rho_*} \\
		&\qquad + C_0 (t^\prime - t)    \mathcal{A}  \Big(\frac{\lambda}{2d \rho_*}+  \frac{a_n A}{n^d |t-t^\prime|^{1/3}} \|H\|_\infty \Big) n^{d-2} g_d (n) \Big\}\\
		&\leq C_5 \Big(\frac{a_n^3}{n^{2d}} + n^{d-2} g_d (n)\Big)
	\end{aligned}
	\end{equation}
		for some constant $C_5 = C_5 (C_0,\lambda,d,\rho_*,H,A,T)$. Above, we used the trivial bound
		\[\int \Gamma_n^{\rm r} (\sqrt{f}) d \nu^n_{\rho_*} \leq C n^d\]
		for any $\nu^n_{\rho_*}$-density $f$. Finally, by \eqref{eqn 4} and \eqref{eqn 5}, the limit in \eqref{eqn 3} is bounded by
		\[\lim_{n \rightarrow \infty} C_6 \Big( \frac{a_n}{n^d} + \frac{n^{2d-2} g_d (n)}{a_n^2}\Big) = 0\]
		for some constant $C_6 = C_6 (C_0,\lambda,d,\rho_*,H,A,T)$, thus concluding the proof.
\end{proof}

By Lemmas \ref{lem: super-esponential estimate 1} and \ref{lem: super-esponential estimate 2}, we can replace the first term inside the brace on the right-hand side of \eqref{eqn 6} by $\<\mu^n_s, (\Delta + F^\prime (\rho_*)) H_s\>$,  replace the second one by $\chi (\rho_*) \sum_{i=1}^d \|\partial_{u_i} H_s\|_{L^2 (\T^d)}^2$,  and the last one by $\tfrac{G(\rho_*)}{2} \|H_s\|_{L^2 (\T^d)}^2$.
Moreover, the time integral of the third one  is super-exponentially small.  Therefore, we can rewrite the exponential martingale as 
\begin{equation}\label{exp martingale}
	\begin{aligned}
		\mathcal{M}^n_t (H) = \exp \Big\{&  \frac{a_n^2}{n^d}  \Big( \<\mu^n_t,H_t\> - \<\mu^n_0,H_0\> - \int_0^t \<\mu^n_s,(\partial_s + \Delta + F^\prime (\rho_*)) H_s\>  ds  \\
		&- 	 \chi (\rho_*) \sum_{i=1}^d \|\partial_{u_i} H\|_{L^2 ([0,t] \times \T^d)}^2 - \frac{G(\rho_*)}{2} \|H\|_{L^2 ([0,t] \times \T^d)}^2 + \int_0^t R^n _s (H) ds\Big) \Big\},
	\end{aligned}
\end{equation} 
where  for any $\varepsilon > 0$,
\begin{equation}\label{error}
	\lim_{n \rightarrow \infty} \frac{n^d}{a_n^2} \log \P^n_{{\rho_*}} \Big( \sup_{0 \leq t \leq T} \Big| \int_0^t R^n _s (H) ds \Big| > \varepsilon \Big) = - \infty.
\end{equation}

\subsection{Exponential tightness}\label{subsec exp tight} In this section, we prove the sequence of processes $\{\mu^n_t, 0 \leq t \leq T\}_{n \geq 1}$ is exponentially tight.  It suffices to prove the following two estimates: for any $H \in \mathcal{D} (\T^d)$, for any $\varepsilon > 0$,
\begin{align}
	\lim_{M \rightarrow \infty} \limsup_{n \rightarrow \infty} \frac{n^d}{a_n^2} \log \P^n_{\rho_*} \Big( \sup_{0 \leq t \leq T} \big|\<\mu^n_t,H\>\big| > M\Big) = - \infty,\label{exp tight c1}\\
	\lim_{\delta \rightarrow 0} \limsup_{n \rightarrow \infty} \frac{n^d}{a_n^2} \log \P^n_{\rho_*} \Big( \sup_{|t-s| \leq \delta} \big|\<\mu^n_t - \mu^n_s,H\>\big| > \varepsilon\Big) = - \infty.\label{exp tight c2}
\end{align}
Recall \eqref{exp martingale}. Together with \eqref{error}, it is easy to see that the three terms in the second line in \eqref{exp martingale} satisfy the above two estimates.  Thus, we only need to respectively consider the following three terms
\[\<\mu^n_0,H\>, \quad \frac{n^d}{a_n^2} \log \mathcal{M}^n_t (H), \quad  \int_0^t \<\mu^n_s,( \Delta + F^\prime (\rho_*)) H_s\>  ds.\]

For the initial term, since $\nu^n_{\rho_*}$ is a product measure, by Taylor's expansion,
\begin{equation}\label{initial exp}
	\begin{aligned}
		\lim_{n \rightarrow \infty} \frac{n^d}{a_n^2} \log \E^n_{\rho_*} \Big[ \exp \Big\{ \frac{a_n^2}{n^d}\<\mu^n_0,H \>  \Big\}\Big] &= 	\lim_{n \rightarrow \infty} \frac{n^d}{a_n^2} \sum_{x \in \T_n^d} \log E_{\nu^n_{\rho_*}} \Big[ \exp \Big\{  \frac{a_n}{n^d} \bar{\eta}_x H (\tfrac{x}{n}) \Big\} \Big] \\
		&= \lim_{n \rightarrow \infty} \frac{n^d}{a_n^2} \sum_{x \in \T_n^d} \log \Big(1+ \frac{a_n^2 \chi (\rho_*)}{2n^{2d}} H^2 \big(\tfrac{x}{n}\big) + \mathcal{O}_H \big(\frac{a_n^3}{n^{3d}}\big)\Big)\\
		&= \frac{\chi(\rho_*)}{2} \|H\|_{L^2 (\T^d)}^2.
	\end{aligned}
\end{equation}
By Markov's inequality, the term $\<\mu^n_0,H\>$ satisfies the estimate in \eqref{exp tight c1}.

Now, we deal with the martingale term.  For the estimate in \eqref{exp tight c1}, by Markov's inequality, we only need to show that, there exists some constant $C$ such that
\[\limsup_{n \rightarrow \infty} \frac{n^d}{a_n^2} \log \E^n_{\rho_*} \Big[ \sup_{0 \leq t \leq T} \exp \{ | \log \mathcal{M}^n_t (H) | \}\Big] \leq C.\]
As before, we can remove the above absolute value, and only need to prove that
\[\limsup_{n \rightarrow \infty} \frac{n^d}{a_n^2} \log \E^n_{\rho_*} \Big[ \sup_{0 \leq t \leq T} \mathcal{M}^n_t (H) \Big] \leq C.\]
By Doob's inequality and \eqref{exp martingale}, the last limit is bounded by
\[\limsup_{n \rightarrow \infty} \frac{n^d}{a_n^2} \log \E^n_{\rho_*} \Big[ \big( \mathcal{M}^n_T (H) \big)^2\Big] \leq \limsup_{n \rightarrow \infty} \frac{n^d}{a_n^2} \log \E^n_{\rho_*} \Big[ \mathcal{M}^n_T (2H) \exp \Big\{ \frac{C(H,\rho_*)a_n^2}{n^d} \Big\}\Big] \leq C (H,\rho_*).\]
For the estimate in \eqref{exp tight c2}, we need to show that, for any $\varepsilon > 0$,
\begin{equation}\label{mart c2}
\lim_{\delta \rightarrow 0} \limsup_{n \rightarrow \infty} \frac{n^d}{a_n^2} \log \P^n_{\rho_*} \Big( \sup_{|t-s| \leq \delta, \atop 0 \leq s,t \leq T} \Big|\frac{n^d}{a_n^2} \log \frac{\mathcal{M}^n_t (H)}{\mathcal{M}^n_s (H)}\Big| > \varepsilon\Big) = - \infty.
\end{equation}
Since
\[\sup_{0 \leq t \leq T} \Big|\frac{n^d}{a_n^2} \log \frac{\mathcal{M}^n_t (H)}{\mathcal{M}^n_{t-} (H)}\Big| \leq \frac{C(H)}{a_n},\]
we have 
\[\Big\{ \sup_{|t-s| \leq \delta, \atop 0 \leq s,t \leq T} \Big|\frac{n^d}{a_n^2} \log \frac{\mathcal{M}^n_t (H)}{\mathcal{M}^n_s (H)}\Big| > \varepsilon\Big\} \subset \bigcup_{k=0}^{[T/\delta]} \Big\{  \sup_{k \delta \leq t \leq (k+1)\delta} \Big|\frac{n^d}{a_n^2} \log \frac{\mathcal{M}^n_t (H)}{\mathcal{M}^n_{k\delta} (H)}\Big| > \varepsilon/4\Big\}.\]
Thus, we only need to show that
\[\lim_{\delta \rightarrow 0}\, \max_{0 \leq k \leq [T/\delta]} \,\limsup_{n \rightarrow \infty} \frac{n^d}{a_n^2} \log \P^n_{\rho_*} \Big( \sup_{k \delta \leq t \leq (k+1)\delta} \Big|\frac{n^d}{a_n^2} \log \frac{\mathcal{M}^n_t (H)}{\mathcal{M}^n_{k\delta} (H)}\Big| > \varepsilon/4\Big) = - \infty.\]
By Markov's inequality and Doob's inequality, it suffices to show that, for any $A > 0$, 
\begin{align*}
&\lim_{\delta \rightarrow 0}\, \max_{0 \leq k \leq [T/\delta]} \,\limsup_{n \rightarrow \infty} \frac{n^d}{a_n^2} \log \E^n_{\rho_*} \Big[  \sup_{k \delta \leq t \leq (k+1)\delta} \Big| \frac{\mathcal{M}^n_t (H)}{\mathcal{M}^n_{k\delta} (H)}\Big|^A \Big] \\
\leq& \lim_{\delta \rightarrow 0}\, \max_{0 \leq k \leq [T/\delta]} \,\limsup_{n \rightarrow \infty} \frac{n^d}{a_n^2} \log \E^n_{\rho_*} \Big[   \Big| \frac{\mathcal{M}^n_{(k+1) \delta} (H)}{\mathcal{M}^n_{k\delta} (H)}\Big|^A \Big] = 0.
\end{align*}
We finish the proof by observing that
\[\E^n_{\rho_*} \Big[  \Big| \frac{\mathcal{M}^n_{(k+1) \delta} (H)}{\mathcal{M}^n_{k\delta} (H)}\Big|^A \Big] \leq \E^n_{\rho_*} \Big[  \frac{\mathcal{M}^n_{(k+1) \delta} (AH)}{\mathcal{M}^n_{k\delta} (AH)} \exp \Big\{  \frac{C(H,A) \delta a_n^2}{n^d}  \Big\} \Big] = \exp \Big\{  \frac{C(H,A) \delta a_n^2}{n^d}  \Big\}.\]

For the term $ \int_0^t \<\mu^n_s,( \Delta + F^\prime (\rho_*)) H\>  ds$, we have the following result.

\begin{lemma}\label{lem: estimate for  time int }
For any $H \in \mathcal{D} (\T^d)$,  and for any $\varepsilon > 0$,
\begin{align}
	\lim_{M \rightarrow \infty} \limsup_{n \rightarrow \infty} \frac{n^d}{a_n^2} \log \P^n_{\rho_*} \Big( \sup_{0 \leq t \leq T} \big|\int_0^t \<\mu^n_r,H\> dr \big| > M\Big) = - \infty,\label{exp tight c3}\\
	\lim_{\delta \rightarrow 0} \limsup_{n \rightarrow \infty} \frac{n^d}{a_n^2} \log \P^n_{\rho_*} \Big( \sup_{|t-s| \leq \delta} \big|\int_s^t \<\mu^n_r,H\> d r\big| > \varepsilon\Big) = - \infty.\label{exp tight c4}
\end{align}
\end{lemma}

\begin{proof}
	We only prove the second identity and the first one could be proved in the same way.  As in the proof of Lemma \ref{lem: super-esponential estimate 1}, by  Garsia-Rodemich-Rumsey inequality (see Lemma \ref{lem: GRR inequality}), 
	\[\sup_{|t-s| \leq \delta} \big|\int_s^t \<\mu^n_r,H\> d r\big| \leq C_1 (q) \delta^{\tfrac{1}{3} - \tfrac{1}{q}} B^{\tfrac{1}{2q}},\]
	where $q > 3$ is fixed, and
	\[B= \int_0^T dt \int_0^T dt^\prime \Big( \frac{1}{|t-t^\prime|^{1/3}} \Big| \int_t^{t^\prime} \<\mu^n_r,H\> dr\Big|\Big).\]
	Similarly to the proof of \eqref{eqn 2}, we only need to prove that, there exists some constant $C$ such that  
	\[	\lim_{n \rightarrow \infty} \frac{n^d}{a_n^2} \log \Big(  \frac{1}{T^2}  \int_0^T dt \int_0^T dt^\prime \,\E^n_{\rho_*} \Big[  \exp \Big\{ \frac{a_n A}{n^d |t-t^\prime|^{1/3} } \big| \int_t^{t^\prime} \sum_{x \in \T_n^d} \bar{\eta}_x (s) H (\tfrac{x}{n}) ds \big|  \Big\}\Big]\Big) \leq C.\]
	Note that the above expression is the same as the one in \eqref{eqn 2} with $A = 1, r_n =a_n$.  By \eqref{eqn 7}, the limit in the last inequality is bounded by 
	\[\lim_{n \rightarrow \infty} C(C_0,d,T,\lambda,\rho_*,\varepsilon_0) \Big(\frac{n^{2d-2} g_d (n)}{a_n^2} + 1\Big) \leq  C(C_0,d,T,\lambda,\rho_*,\varepsilon_0),\]
	thus concluding the proof.
\end{proof}

\subsection{Proof of the upper bound}\label{subsec pf upper} By the exponential tightness of the sequence $\{\mu^n_\cdot\}$, we only need to prove that for any compact set $K \subset D([0,T],\mathcal{D}^\prime (\T^d))$, 
\begin{equation}\label{eqn 9}
	\limsup_{n \rightarrow \infty} \frac{n^d}{a_n^2} \log \P^n_{\rho_*} \Big(\mu^n_\cdot \in K\Big) \leq  - \inf_{\mu \in K} \mathcal{Q}_T (\mu).
\end{equation}
For any $\delta > 0$, let $\mathcal{B}_\delta$ be the event that 
\[\Big|\int_0^T R^n_s (H) ds \Big| \leq \delta.\]
Then, by \eqref{error} and \eqref{exp martingale}, the $\limsup$ in \eqref{eqn 9} is bounded by
\begin{align*}
	&\limsup_{n \rightarrow \infty} \frac{n^d}{a_n^2} \log \P^n_{\rho_*} \Big(\mu^n_\cdot \in K, \mathcal{B}_\delta\Big)\\
	&= \limsup_{n \rightarrow \infty} \frac{n^d}{a_n^2} \log \E^n_{\rho_*} \Big[ \mathcal{M}^n_T (H)^{-1} \mathcal{M}^n_H (H) \mathbf{1} \{ \mu^n_\cdot \in K, \mathcal{B}_\delta \} \Big]\\
	& \leq - \inf_{\mu \in K} \Big\{\ell_T (\mu,H) - 	 \chi (\rho_*) \sum_{i=1}^d \|\partial_{u_i} H\|_{L^2 ([0,T] \times \T^d)}^2 - \frac{G(\rho_*)}{2} \|H\|_{L^2 ([0,T] \times \T^d)}^2 + \langle \mu_0,\phi \rangle\Big\} \\
	&+ \limsup_{n \rightarrow \infty} \frac{n^d}{a_n^2} \log \E^n_{\rho_*} \Big[\exp \Big\{ \langle \mu^n_0,\phi \rangle\Big\}\Big] + \delta
\end{align*}
for any $\phi \in \mathcal{D} (\T^d)$.  By \eqref{initial exp}, the $\limsup$ in the last line equals $\tfrac{\chi (\rho_*)}{2} \|\phi\|_{L^2 (\T^d)}^2$.  Letting $\delta \rightarrow 0$ and then optimizing over $H \in C^{1,\infty} ([0,T] \times \R^d), \phi \in \mathcal{D} (\T^d)$, we have 
\begin{multline*}
\limsup_{n \rightarrow \infty} \frac{n^d}{a_n^2} \log \P^n_{\rho_*} \Big(\mu^n_\cdot \in K\Big) \leq  - \sup_{H \in C^{1,\infty} ([0,T] \times \R^d), \phi \in \mathcal{D} (\T^d)} \inf_{\mu \in K} \Big\{\ell_T (\mu,H) \\
- 	 \chi (\rho_*) \sum_{i=1}^d \|\partial_{u_i} H\|_{L^2 ([0,T] \times \T^d)}^2 
- \frac{G(\rho_*)}{2} \|H\|_{L^2 ([0,T] \times \T^d)}^2 + \langle \mu_0,\phi \rangle - \frac{\chi (\rho_*)}{2} \|\phi\|_{L^2 (\T^d)}^2\Big\}.
\end{multline*}
Note that the term inside the above brace is continuous in $\mu,H,\phi$, linear in $\mu$ and concave in $H,\phi$.  By the Minimax theorem \cite[Page 584]{gao2003moderate}, we can exchange the order of the above $\sup$ and $\inf$, thus concluding the proof of the upper bound.

\section{The lower bound}\label{sec: lower bound}

In subsection \ref{subsec perturbed dynamics}, we state  hydrodynamic limits  for a perturbed dynamics corresponding to the exponential martingale constructed previously in Subsection \ref{subsec:exp martingale}. The proof of this hydrodynamic limits is presented in Subsection \ref{sec: pf hydrodynamics}. The straightforward approach is to use the entropy method, see \cite[Chapter 5]{klscaling} for example. However, for the entropy method, one of the steps is to prove the absolute continuity of the limiting point with respect to the Lebesgue measure, which we are not aware of how to prove in this case. For this reason, we adopt the relative entropy method, where we need another version of main lemma, see Proposition \ref{lem main lemma 2}.  With the above hydrodynamic limits, we prove the lower bound in Subsection \ref{subsec: pf lower bound}.
 
\subsection{Hydrodynamic limits for the perturbed dynamics}\label{subsec perturbed dynamics} For any $\phi \in \mathcal{D} (\T^d)$, define $\nu^n_{\phi,\rho_*} $ as the product measure on $\Omega_n$ with marginals 
\[\nu^n_{\phi,\rho_*} (\eta_x = 1) = \rho_* + \frac{a_n}{n^d} \phi \Big(\frac{x}{n}\Big), \quad x \in \T_n^d.\]
For any $H \in C^{1,\infty} ([0,T] \times \T^d)$, define
\[ \frac{d \P^n_{H,\phi}}{d  \P^n_{{\rho_*}}}= \mathcal{M}^n_T (H) \frac{d \nu^n_{\phi, \rho_*}}{d \nu^n_{\rho_*}}.\]
Under $\P^n_{H,\phi}$, $(\eta (t))_{t \geq 0}$ is a time-inhomogeneous Markov process on the state space $\Omega_n$ with generator $L_{n,t}^H$ and with initial distribution $\nu^n_{\phi,\rho_*}$,  where for $f: \Omega_n \rightarrow \R$,
\begin{align*}
L_{n,t}^H f (\eta) &= n^2 \sum_{x \in \T_n^d} \sum_{i=1}^d  \exp \big\{   \tfrac{a_n}{n^d} (\eta_x - \eta_{x+e_i}) \big( H_t(\tfrac{x+e_i}{n}) - H_t (\tfrac{x}{n})\big)  \big\}  \big[ f(\eta^{x,x+e_i}) - f(\eta) \big] \\
&+ \sum_{x \in \T_n^d} c_x (\eta) \exp \big\{   \tfrac{a_n}{n^d} (1-2\eta_x) H_t (\tfrac{x}{n})  \big\} \big[ f(\eta^{x}) - f(\eta) \big],
\end{align*}
see \cite[Appendix 1, Proposition 7.3]{klscaling} for example.

The following result concerns hydrodynamic limits for $\mu^n_t$ in the above perturbed dynamics.  It can also be regarded as an equilibrium perturbation of the dynamics, see \cite{xu2023equilibrium} for example.

\begin{proposition}\label{pro: hydrodynamic limits}
Under $\P^n_{H,\phi}$, the sequence $\{\mu^n_t, 0 \leq t \leq T\}_{n \geq 1}$ converges in probability to the deterministic measure $\{\rho(t,u)du, 0 \leq t \leq T\}$, where $\rho (t,u)$ is the unique weak solution to the following PDE:
\begin{equation}\label{hydrodynamic equation}
\begin{cases}
\partial_t  \rho (t,u) = (\Delta + F^\prime (\rho_*)) \rho (t,u) - 2 \chi (\rho_*) \Delta H (t,u) + G(\rho_*) H (t,u), \quad t > 0, u \in \T^d,\\
\rho (0,u) = \phi (u), \quad u \in \T^d.
\end{cases}
\end{equation}
\end{proposition}

\subsection{Proof of the lower bound}\label{subsec: pf lower bound} We first introduce some properties of the rate function $\mathcal{Q}_T$.  For any $H,J \in C^{1,\infty} ([0,T] \times \T^d)$, define the scalar product $[\cdot,\cdot]$ as
\[[H,J] := \chi (\rho_*) \sum_{i=1}^d \int_0^T \int_{\T^d} \partial_{u_i} H (t,u) \partial_{u_i} J (t,u) \,du\,dt + \frac{G(\rho_*)}{2} \int_0^T \int_{\T^d}  H (t,u)  J (t,u) \,du\,dt. \]
Let $\mathcal{H}$ be the Hilbert space defined as the completion of the space $H,J \in C^{1,\infty} ([0,T] \times \T^d)$ with respect to the above scalar product. With the above notation, we have 
\[\mathcal{Q}_{\rm dyn} (\mu) := \sup_{J \in C^{1,\infty} ([0,T] \times \T^d)} \Big\{ \ell_T (\mu,J) -[J,J]\Big\}.\]
Also recall \eqref{q 0}. Then, by Riesz representation theorem, it is easy to prove the following result. We omit the proof here and refer the readers to \cite[Lemma 5.1]{gao2003moderate} for details.

\begin{lemma}\label{lem properties rate function}
If $\mathcal{Q}_T (\mu) < + \infty$, then there exist $\phi \in L^2 (\T^d)$ and $H \in \mathcal{H}$ such that
\[\<\mu_0,\varphi\> = \<\phi,\varphi\>,\;\forall \varphi \in \mathcal{D} (\T^d); \quad \mathcal{Q}_0 (\mu_0) = \frac{1}{2 \chi(\rho_*)} \|\phi\|_{L^2 (\T^d)}^2,\]
and that
\[\ell_T (\mu, J) = 2 [H,J], \;\forall J \in C^{1,\infty} ([0,T] \times \R^d); \quad \mathcal{Q}_{\rm dyn} (\mu) = [H,H].\]
\end{lemma}

Now, we prove the lower bound. For any open set $O \in D([0,T],\mathcal{D}^\prime (\T^d))$, it suffices to show that for any $\mu \in O$,
\begin{equation}
	\liminf_{n \rightarrow \infty} \frac{n^d}{a_n^2} \log \P^n_{\rho_*} \Big( \mu^n_\cdot \in O \Big) \geq - \mathcal{Q}_T (\mu).
\end{equation}
Using an approximation procedure, we can assume that $\mu (t,du) = \rho (t,u) du $ with $ \rho (t,u) \in C^{1,\infty} ([0,T] \times \T^d)$ and that $H \in C^{1,\infty} ([0,T] \times \T^d), \phi \in \mathcal{D} (\T^d)$, where $H$ and $\phi$ are identified in Lemma \ref{lem properties rate function}, see \cite{gao2003moderate} for example.  Moreover, $\rho(t,u)$ is the unique weak solution to \eqref{hydrodynamic equation}.  By Jensen's inequality,
\begin{equation}\label{eqn 8}
\begin{aligned}
	&\frac{n^d}{a_n^2} \log \P^n_{\rho_*} \Big( \mu^n_\cdot \in O \Big) = \frac{n^d}{a_n^2}  \log \E^n_{H,\phi} \Big[ \mathcal{M}^n_T (H)^{-1} \frac{d \nu^n_{\rho_*}}{d \nu^n_{\phi,\rho_*}} \mathbf{1} \{ \mu^n_\cdot \in O\}\Big]\\
	&\geq \frac{n^d}{a_n^2}  \log  \P^n_{\rho_*} \Big( \mu^n_\cdot \in O\Big) + \frac{n^d}{a_n^2}   \E^n_{H,\phi} \Big[ \log \mathcal{M}^n_T (H)^{-1}  \Big| \mu^n_\cdot \in O \Big] + \frac{n^d}{a_n^2}   \E^n_{H,\phi} \Big[ \log  \frac{d \nu^n_{\rho_*}}{d \nu^n_{\phi,\rho_*}}  \Big| \mu^n_\cdot \in O \Big].
\end{aligned}
\end{equation}
By Proposition \ref{pro: hydrodynamic limits},
\[\lim_{n \rightarrow \infty} \P^n_{\rho_*} \Big( \mu^n_\cdot \in O\Big)  = 1.\]
In particular, 
\[\lim_{n \rightarrow \infty} \frac{n^d}{a_n^2}  \log  \P^n_{\rho_*} \Big( \mu^n_\cdot \in O\Big)  = 0.\]
To deal with the remaining two terms on the right-hand side of \eqref{eqn 8}, we need the following lemma, which can be proved by using the entropy inequality, see \cite[Lemma 4.3]{zhao2024moderate} for example.

\begin{lemma}\label{lem 1}
	Let $\mathcal{A} \in D([0,T],\Omega_n)$.  If
	\[\lim_{n \rightarrow \infty} \frac{n^d}{a_n^2} \log \P^n_{\rho_*} (\mathcal{A}) = - \infty,\]
	then,
	\[\lim_{n \rightarrow \infty} \P^n_{H,\phi} (\mathcal{A}) = 0.\]
\end{lemma}

For the martingale term on the right-hand side of \eqref{eqn 8}, 
\begin{align*}
	\lim_{n \rightarrow \infty} \frac{n^d}{a_n^2} \log  \mathcal{M}^n_T (H)^{-1} = - \lim_{n \rightarrow \infty} \Big\{  \ell_T (\mu^n,H) - [H,H] \Big\} = - \Big\{  \ell_T (\mu,H) - [H,H] \Big\} = - \mathcal{Q}_{\rm dyn} (\mu)  
\end{align*}
in $\P^n_{H,\phi}$-probability.  In the above first identity we used \eqref{exp martingale}, \eqref{error} and Lemma \ref{lem 1}; in the second one we used Proposition \ref{pro: hydrodynamic limits}; in the last one we used Lemma \ref{lem properties rate function}. By Proposition \ref{pro: hydrodynamic limits} and dominated convergence theorem,
\[\lim_{n \rightarrow \infty} \frac{n^d}{a_n^2}   \E^n_{H,\phi} \Big[ \log \mathcal{M}^n_T (H)^{-1}  \Big| \mu^n_\cdot \in O \Big] =  - \mathcal{Q}_{\rm dyn} (\mu) .\]

For the last term on the right-hand side of \eqref{eqn 8},  since $\nu^n_{\rho_*}$ and $\nu^n_{\phi,\rho_*}$ are both product measures, direct calculations yield that 
\[\lim_{n \rightarrow \infty} \frac{n^d}{a_n^2}   \E^n_{H,\phi} \Big[ \log  \frac{d \nu^n_{\rho_*}}{d \nu^n_{\phi,\rho_*}}  \Big| \mu^n_\cdot \in O \Big] = - \frac{1}{2 \chi(\rho_*)} \|\phi\|_{L^2 (\T^d)}^2 = - \mathcal{Q}_0 (\mu_0),\]
see \cite{zhao2024moderate} for details. 

Finally, we conclude the proof of the lower bound by \eqref{eqn 8} and the above estimates.

\subsection{Proof of Proposition \ref{pro: hydrodynamic limits}}\label{sec: pf hydrodynamics} 

Let $\pi^n_t$ be the distribution of the process at time $t$, and introduce the reference measure $\nu^n_t$ as the product measure on $\Omega_n$ with marginals
\[\nu^n_t (\eta_x = 1) = \rho^n_x, \quad x \in \T^d_n,\]
where we shorten \[\rho^n_x := \rho^n (t,\tfrac{x}{n}) := \rho_* + \frac{a_n}{n^d} \rho (t,\tfrac{x}{n}).\]
Recall the relative entropy of $\pi^n_t$ with respect to $\nu^n_t$ is defined as 
\[\mathcal{H}_n (t) := \mathcal{H} (\pi^n_t | \nu^n_t) = \int f^n_t \log f^n_t d \nu^n_t, \quad f^n_t = d \pi^n_t / d \nu^n_t.\]
The main aim of this section is to prove the following result.

\begin{proposition}\label{pro: relative entropy}
For any $0 \leq t \leq T$, $\mathcal{H}_n (t) \ll a_n^2 / n^d$.
\end{proposition}

We first prove the hydrodynamic limits from the above result.

\begin{proof}[Proof of Proposition \ref{pro: hydrodynamic limits} from Proposition \ref{pro: relative entropy}]
By the entropy inequality, it is directly from the above proposition that, for any $0 \leq t \leq T$, for any $H \in \mathcal{D} (\T^d)$,
\[\lim_{n \rightarrow \infty} \< \mu^n_t, H \> = \int_{\T^d} \rho (t,u) H(u) du \]
in $\P^n_{H,\phi}$-probability, see \cite[Proof of Theorem 2.2]{xu2023equilibrium} for example.  Thus, to show the sample-path convergence, we only need to show $\{\mu^n_t, 0 \leq t \leq T\}$ is tight in the space $D ([0,T], \mathcal{D}^\prime(\mathbb{T}^d))$. It suffices to show that, for any $H \in \mathcal{D} (\T^d)$  and for any $\varepsilon > 0$,
\begin{align*}
	\lim_{M \rightarrow \infty} \limsup_{n \rightarrow \infty} \P^n_{H,\phi} \Big( \sup_{0 \leq t \leq T} \big| \<\mu^n_t,H\>  \big| > M\Big) = 0,\\
	\lim_{\delta \rightarrow 0} \limsup_{n \rightarrow \infty} \P^n_{H,\phi} \Big( \sup_{|t-s| \leq \delta} \big| \<\mu^n_t - \mu^n_s,H\> \big| > \varepsilon\Big) = 0.
\end{align*}
The above two estimates follows immediately from \eqref{exp tight c1}, \eqref{exp tight c2} and Lemma \ref{lem 1}.
\end{proof}

Now, we prove Proposition \ref{pro: relative entropy}. By Yau's relative entropy inequality \cite{jara2018non},
\[\mathcal{H}_n^\prime (t) \leq \int \Big(L_{n,t}^{H,*} \mathbf{1} - \partial_t \log \nu^n_t\Big) f^n_t d \nu^n_t- \int \Gamma_{n,t}^H (\sqrt{f^n_t}) d \nu^n_t,\]
where $L_{n,t}^{H,*}$ is the adjoint of $L_{n,t}^{H}$ in $L^2 (\nu^n_t)$, and $\Gamma_{n,t}^H (f) = \Gamma_{n,t}^{H,{\rm ex}} (f) + \Gamma_{n,t}^{H,{\rm r}} (f)$, where
\begin{align*}
\Gamma_{n,t}^{H,{\rm ex}} (f)  &= \frac{n^2}{2} \sum_{x \in \T_n^d} \sum_{i=1}^d  \exp \big\{   \tfrac{a_n}{n^d} (\eta_x - \eta_{x+e_i}) \big( H_t(\tfrac{x+e_i}{n}) - H_t (\tfrac{x}{n})\big)  \big\}  \big[ f(\eta^{x,x+e_i}) - f(\eta) \big]^2, \\
\Gamma_{n,t}^{H,{\rm r}} (f) &= \frac{1}{2} \sum_{x \in \T_n^d} c_x (\eta) \exp \big\{   \tfrac{a_n}{n^d} (1-2\eta_x) H_t (\tfrac{x}{n})  \big\} \big[ f(\eta^{x}) - f(\eta) \big]^2.
\end{align*}
By Lemma \ref{lem 2}, $L_{n,t}^{H,*} \mathbf{1} = A_n (t) + B_n (t)$,  where 	$A_n (t)$
corresponds to the dynamics of the exclusion process and equals
\begin{align*}
 &n^2 \sum_{x \in \T_n^d} \sum_{i=1}^d \Big\{ \eta_x (1-\eta_{x+e_i}) \Big[  \exp \Big\{ - \frac{a_n}{n^{d+1}} \nabla_{n,i} H_t (\tfrac{x}{n}) \Big\} \frac{\rho^n_{x+e_i} (1-\rho^n_x)}{\rho^n_x (1-\rho^n_{x+e_i})} - \exp \Big\{ \frac{a_n}{n^{d+1}} \nabla_{n,i} H_t (\tfrac{x}{n}) \Big\}\Big] \\
	&+ \eta_{x+e_i} (1-\eta_{x}) \Big[  \exp \Big\{  \frac{a_n}{n^{d+1}} \nabla_{n,i} H_t (\tfrac{x}{n}) \Big\} \frac{\rho^n_{x} (1-\rho^n_{x+e_i})}{\rho^n_{x+e_i} (1-\rho^n_{x})} - \exp \Big\{- \frac{a_n}{n^{d+1}} \nabla_{n,i} H_t (\tfrac{x}{n}) \Big\}\Big] \Big\},
\end{align*}
and the Glauber part $B_n (t)$ equals 
\begin{align*}
& \sum_{x \in \T_n^d} \Big\{    \eta_x \Big[ \Big(a+\frac{\lambda}{2d} \sum_{y \sim x} \eta_y \Big) \exp \Big\{  \frac{a_n}{n^d} H_t (\tfrac{x}{n}) \Big\}  \frac{1-\rho^n_x}{\rho^n_x}- b \exp \Big\{ - \frac{a_n}{n^d} H_t (\tfrac{x}{n})\Big\}\Big]  \\
	&+ (1-\eta_x) \Big[ b \exp \Big\{ - \frac{a_n}{n^d} H_t (\tfrac{x}{n})\Big\}  \frac{\rho^n_x}{1-\rho^n_x} - \Big(a+\frac{\lambda}{2d} \sum_{y \sim x} \eta_y \Big)   \exp \Big\{  \frac{a_n}{n^d} H_t (\tfrac{x}{n}) \Big\}  \Big]\Big\}.
\end{align*}
The time derivative equals
\begin{align*}
	\partial_t \log \nu^n_t = \partial_t \sum_{x \in \T_n^d} \Big[(1-\eta_x) \log (1 - \rho^n_x) + \eta_x \log \rho^n_x\Big] =  \frac{a_n}{n^d}\sum_{x \in \T_n^d} w_x \partial_t \rho(t,\tfrac{x}{n}),
\end{align*}
where
\[w_x := w^n(t,x; \eta) := \frac{\eta_x - \rho^n_x}{\rho^n_x (1-\rho^n_x)}.\]

Next, we deal with the term $A_n (t)$. By Taylor's expansion, we write $A_n (t) = \sum_{i=1}^3 A_{n,i} (t) + \mathcal{O}_H (a_n^3 / n^{2d+1})$, where
\begin{align*}
	A_{n,1} (t) &= n^2 \sum_{x \in \T_n^d} \sum_{i=1}^d \Big\{ \eta_x (1-\eta_{x+e_i}) \Big[  \frac{\rho^n_{x+e_i} (1-\rho^n_x)}{\rho^n_x (1-\rho^n_{x+e_i})} - 1\Big] 
	+ \eta_{x+e_i} (1-\eta_{x}) \Big[   \frac{\rho^n_{x} (1-\rho^n_{x+e_i})}{\rho^n_{x+e_i} (1-\rho^n_{x})} -1\Big] \Big\},\\
	A_{n,2} (t) &=  \frac{a_n n^2}{n^{d+1}} \sum_{x \in \T_n^d} \sum_{i=1}^d \nabla_{n,i} H_t (\tfrac{x}{n}) \Big\{ \eta_x (1-\eta_{x+e_i}) \Big[ - \frac{\rho^n_{x+e_i} (1-\rho^n_x)}{\rho^n_x (1-\rho^n_{x+e_i})} - 1\Big] \\
	&\qquad \qquad + \eta_{x+e_i} (1-\eta_{x}) \Big[   \frac{\rho^n_{x} (1-\rho^n_{x+e_i})}{\rho^n_{x+e_i} (1-\rho^n_{x})} +1\Big] \Big\},\\
	A_{n,3} (t) &=   \frac{a_n^2 n^2}{2n^{2d+2}} \sum_{x \in \T_n^d} \sum_{i=1}^d \Big( \nabla_{n,i} H_t (\tfrac{x}{n})\Big)^2 \Big\{ \eta_x (1-\eta_{x+e_i}) \Big[  \frac{\rho^n_{x+e_i} (1-\rho^n_x)}{\rho^n_x (1-\rho^n_{x+e_i})} - 1\Big] \\
	&\qquad \qquad + \eta_{x+e_i} (1-\eta_{x}) \Big[   \frac{\rho^n_{x} (1-\rho^n_{x+e_i})}{\rho^n_{x+e_i} (1-\rho^n_{x})} -1\Big] \Big\}.
\end{align*}  
Note that $\mathcal{O}_H (a_n^3 / n^{2d+1}) \ll a_n^2/n^d$. For $A_{n,1} (t)$, we first write the term inside the brace as
\begin{align*}
(\rho^n_{x+e_i} - \rho^n_x) \Big[\frac{\eta_x (1-\eta_{x+e_i}) }{\rho^n_x (1-\rho^n_{x+e_i})} - \frac{\eta_{x+e_i} (1-\eta_{x})}{\rho^n_{x+e_i} (1-\rho^n_{x})}\Big] = (\rho^n_{x+e_i} - \rho^n_x) \big[w_x - w_{x+e_i} - (\rho^n_{x+e_i} - \rho^n_x) w_x w_{x+e_i}\big],
\end{align*}
then, using the summation by parts formula,
\begin{align*}
A_{n,1} (t) = \frac{a_n}{n^d} \sum_{x \in \T_n^d} w_x \Delta_n \rho (t,\tfrac{x}{n}) - \frac{a_n^2}{n^{2d}} \sum_{x \in \T_n^d} \sum_{i=1}^d\Big( \nabla_{n,i} \rho (t,\tfrac{x}{n})\Big)^2 w_x w_{x+e_i}.
\end{align*}
By  Lemmas \ref{lem: super-esponential estimate 2} and \ref{lem 1}, it is easy to see that
\begin{equation}\label{error small}
\int_0^t E_{\pi^n_s} \Big[\frac{a_n^2}{n^{2d}} \sum_{x \in \T_n^d} \sum_{i=1}^d\Big( \nabla_{n,i} \rho (s,\tfrac{x}{n})\Big)^2 w_x w_{x+e_i}\Big] ds \ll \frac{a_n^2}{n^d}.
\end{equation}
Similarly, since $|\rho^n_{x+e_i} - \rho^n_x| \leq Ca_n/n^{d+1}$,
\[|A_{n,3} (t)|  \leq \frac{C(H) a_n^2}{n^{2d}} \times n^d \times \frac{a_n}{n^{d+1}} = \frac{C(H) a_n^3}{n^{2d+1}} \ll \frac{a_n^2}{n^d}.\]
For $A_{n,2} (t)$, we write the term inside the brace as 
 \begin{align*}
 &\Big[ \rho^n_{x+e_i} (1-\rho^n_x) + \rho^n_{x} (1-\rho^n_{x+e_i})  \Big] \Big[ 	 \frac{\eta_{x+e_i} (1-\eta_{x})}{\rho^n_{x+e_i} (1-\rho^n_{x})}  -  \frac{\eta_x (1-\eta_{x+e_i})}{\rho^n_x (1-\rho^n_{x+e_i})}  \Big] \\
 =&  \Big[ \rho^n_{x+e_i} (1-\rho^n_x) + \rho^n_{x} (1-\rho^n_{x+e_i})  \Big]  \Big[ w_{x+e_i} - w_x + (\rho^n_{x+e_i} - \rho^n_x) w_x w_{x+e_i}\Big].
 \end{align*}
 Thus, 
 \begin{align*}
 		A_{n,2} (t) &=  \frac{a_n }{n^{d-1}} \sum_{x \in \T_n^d} \sum_{i=1}^d \nabla_{n,i} H_t (\tfrac{x}{n}) \Big[ \rho^n_{x+e_i} (1-\rho^n_x) + \rho^n_{x} (1-\rho^n_{x+e_i})  \Big]  (w_{x+e_i} - w_x)\\
 		&+ \frac{a_n^2}{n^{2d}}  \sum_{x \in \T_n^d} \sum_{i=1}^d \nabla_{n,i} H_t (\tfrac{x}{n}) \Big[ \rho^n_{x+e_i} (1-\rho^n_x) + \rho^n_{x} (1-\rho^n_{x+e_i})  \Big]  \nabla_{n,i} \rho(t,\tfrac{x}{n}) w_x w_{x+e_i}.
 \end{align*}
By  Lemmas \ref{lem: super-esponential estimate 2} and \ref{lem 1}, the last line satisfies the estimate in \eqref{error small}. For the first line in the expression of $A_{n,2} (t)$,  first replacing  $\rho^n_{x+e_i} (1-\rho^n_x) + \rho^n_{x} (1-\rho^n_{x+e_i})$ with $2 \chi(\rho_*)$ plus an error term of order $a_n/n^d$, then using the summation by parts formula, and finally by Lemmas \ref{lem: super-esponential estimate 1} and \eqref{lem 1},
\[A_{n,2} (t) =  -2 \chi (\rho_*) \frac{a_n }{n^{d}} \sum_{x \in \T_n^d} w_x \Delta_n H_t (\tfrac{x}{n}) \]
plus an error term satisfying the estimate in \eqref{error small}. To sum up, we have shown that
\[A_n (t) =  \frac{a_n}{n^d} \sum_{x \in \T_n^d} w_x \Big( \Delta_n \rho (t,\tfrac{x}{n}) - 2 \chi (\rho_*) \Delta_n H_t (\tfrac{x}{n}) \Big)\]
plus an error term satisfying the estimate in \eqref{error small}. 

Now, we deal with the term $B_n (t)$. By Taylor's expansion, $B_n (t) = \sum_{i=1}^3 B_{n,i} (t) + \mathcal{O}_H (a_n^3/n^{2d})$, where
\begin{align*}
B_{n,1} (t) &= 	 \sum_{x \in \T_n^d} \Big\{    \eta_x \Big[ \Big(a+\frac{\lambda}{2d} \sum_{y \sim x} \eta_y \Big) \frac{1-\rho^n_x}{\rho^n_x}- b \Big]  + (1-\eta_x) \Big[ b   \frac{\rho^n_x}{1-\rho^n_x} - \Big(a+\frac{\lambda}{2d} \sum_{y \sim x} \eta_y \Big)   \Big]\Big\},\\
B_{n,2} (t) &= 	\frac{a_n}{n^d} \sum_{x \in \T_n^d} H_t (\tfrac{x}{n})\Big\{    \eta_x \Big[ \Big(a+\frac{\lambda}{2d} \sum_{y \sim x} \eta_y \Big) \frac{1-\rho^n_x}{\rho^n_x}+ b \Big]  - (1-\eta_x) \Big[ b   \frac{\rho^n_x}{1-\rho^n_x} + \Big(a+\frac{\lambda}{2d} \sum_{y \sim x} \eta_y \Big)   \Big]\Big\},\\
B_{n,3} (t) &= 	\frac{a_n^2}{2n^{2d}} \sum_{x \in \T_n^d} H_t^2 (\tfrac{x}{n})\Big\{    \eta_x \Big[ \Big(a+\frac{\lambda}{2d} \sum_{y \sim x} \eta_y \Big) \frac{1-\rho^n_x}{\rho^n_x}- b \Big]  + (1-\eta_x) \Big[ b   \frac{\rho^n_x}{1-\rho^n_x} - \Big(a+\frac{\lambda}{2d} \sum_{y \sim x} \eta_y \Big)   \Big]\Big\}.
\end{align*}
For $B_{n,1} (t)$, we write the term inside the brace as
\begin{align*}
	&\Big(\frac{\eta_x}{\rho^n_x} - \frac{1-\eta_x}{1-\rho^n_x}\Big) \Big[ \Big(a+\frac{\lambda}{2d} \sum_{y \sim x} \eta_y \Big)  (1-\rho^n_x) - b \rho^n_x\Big] \\
	&= w_x \Big[F (\rho^n_x) + \frac{\lambda (1-\rho^n_x)}{2d} \sum_{y \sim x} \big\{ (\eta_y - \rho^n_y) +( \rho^n_y - \rho^n_x) \big\} \Big].
\end{align*}
Since $F(\rho^n_x) = F^\prime (\rho_*) a_n \rho (t,\tfrac{x}{n}) / n^d + \mathcal{O} (a_n^2 / n^{2d})$, by Lemmas \ref{lem: super-esponential estimate 1} and \ref{lem 1}, 
\[ \sum_{x \in \T_n^d} w_x F(\rho^n_x) = \frac{a_n}{n^d} \sum_{x \in \T_n^d} F^\prime (\rho_*) w_x \rho (t,\tfrac{x}{n})\]
plus an error term satisfying the estimate in \eqref{error small}.   Similarly, the term 
\[ \frac{\lambda}{2d} \sum_{x \in \T_n^d}  (1-\rho^n_x) w_x \sum_{y \sim x}  ( \rho^n_y - \rho^n_x) = \frac{ \lambda  a_n}{2d n^{d+2}}  \sum_{x \in \T_n^d}  (1-\rho^n_x) w_x \Delta_n \rho (t,\tfrac{x}{n}) \]
also satisfies the estimate in \eqref{error small}.  We are left with the term 
\[ \frac{\lambda }{2d} \sum_{x \in \T_n^d} \sum_{y \sim x}  (1-\rho^n_x) w_x (\eta_y - \rho^n_y) =  \frac{\lambda }{2d} \sum_{x \in \T_n^d} \sum_{y \sim x}  (1-\rho^n_x) \chi (\rho^n_y) w_x w_y.\]
Note that we cannot use Lemma \ref{lem: super-esponential estimate 2} to deal with the above term since $a_n^2 / n^d \ll a_n$.  Instead, we need the following version of the main lemma, see \cite[Lemma 3.1]{jara2018non}.

\begin{proposition}[The second main lemma]\label{lem main lemma 2}
There exists some constant $C=C(\lambda,d,\rho_*)$ such that 
\[\int \frac{\lambda }{2d} \sum_{x \in \T_n^d} \sum_{y \sim x}  (1-\rho^n_x) \chi (\rho^n_y) w_x w_y f^n_t d \nu^n_t \leq \frac{1}{4} \int \Gamma_{n,t}^{H, {\rm ex}} (\sqrt{f^n_t}) d \nu^n_t + C \big(\mathcal{H}_n (t) + n^{d-2} g_d (n)\big).\]
\end{proposition}

Since $n^{d-1} \sqrt{g_d (n)} \ll a_n$, we have
\begin{align*}
\frac{n^d}{a_n^2} \int_0^t E_{\pi^n_s} \big[B_{n,1} (s) \big] ds &\leq \int_0^t E_{\pi^n_s} \Big[ \frac{1}{a_n} \sum_{x \in \T_n^d} F^\prime (\rho_*) w_x \rho (s,\tfrac{x}{n})\Big] ds \\
&+ \frac{1}{4} \int_0^t ds \Big\{ \int \Gamma_{n,s}^{H, {\rm ex}} (\sqrt{f^n_s}) d \nu^n_s \Big\} + \frac{Cn^d}{a_n^2} \int_0^t \mathcal{H}_n (s) ds + o_n (1). 
\end{align*}

Comparing the expression of $B_{n,3} (t)$ with $B_{n,1} (t)$, it is immediately that $B_{n,3} (t)$ satisfies the estimate in \eqref{error small}. It remains to deal with $B_{n,2} (t)$.  We first write the term inside the brace as
\begin{align*}
w_x \Big[\Big(a+\frac{\lambda}{2d} \sum_{y \sim x} \eta_y \Big) (1-\rho_x^n) + b \rho^n_x\Big] = w_x G(\rho^n_x) + \frac{\lambda}{2d} \sum_{y \sim x} w_x (1-\rho^n_x) [\eta_y - \rho^n_x].
\end{align*}
By Taylor's expansion, Lemmas \ref{lem: super-esponential estimate 1} and \ref{lem 1}, 
\begin{align*}
	\frac{a_n}{n^d} \sum_{x \in \T_n^d} H_t (\tfrac{x}{n})  w_x G(\rho^n_x) = \frac{a_n}{n^d} \sum_{x \in \T_n^d} H_t (\tfrac{x}{n})  w_x G(\rho_*) 
\end{align*}
plus an error term satisfying the estimate in \eqref{error small}.  For the remaining term in $B_{n,2} (t)$, we write
\begin{align*}
&\frac{a_n}{n^d} \sum_{x \in \T_n^d} H_t (\tfrac{x}{n})  \sum_{y \sim x} w_x (1-\rho^n_x) [\eta_y - \rho^n_x]  \\
&= \frac{a_n}{n^d} \sum_{x \in \T_n^d} H_t (\tfrac{x}{n}) \sum_{y \sim x} w_x (1-\rho^n_x) [\eta_y - \rho^n_y] + \frac{a_n^2}{n^{2d+2}}\sum_{x \in \T_n^d} H_t (\tfrac{x}{n}) w_x (1-\rho^n_x) \Delta_n \rho (t,\tfrac{x}{n}).
\end{align*}
Note that the first term above can be dealt with by Proposition \ref{lem main lemma 2} and the second one has order $o(a_n^2/n^{d})$. Thus,
\begin{align*}
	\frac{n^d}{a_n^2} \int_0^t E_{\pi^n_s} \big[B_{n,2} (s) \big] ds &\leq \int_0^t E_{\pi^n_s} \Big[ \frac{1}{a_n} \sum_{x \in \T_n^d} G (\rho_*) w_x H_s (\tfrac{x}{n})\Big] ds \\
	&+ \frac{1}{4a_n} \int_0^t ds \Big\{ \int \Gamma_{n,s}^{H, {\rm ex}} (\sqrt{f^n_s}) d \nu^n_s \Big\} + \frac{Cn^d}{a_n^3} \int_0^t \mathcal{H}_n (s) ds + o_n (1). 
\end{align*}
To sum up, we have shown that
\begin{multline*}
	\frac{n^d}{a_n^2} \int_0^t E_{\pi^n_s} \big[B_{n} (s) \big] ds \leq \int_0^t E_{\pi^n_s} \Big[ \frac{1}{a_n} \sum_{x \in \T_n^d} w_x  \Big( F^\prime (\rho_*) \rho(s,\tfrac{x}{n}) + G (\rho_*)  H_s (\tfrac{x}{n}) \Big) \Big] ds \\
	+ \frac{1}{4} \Big(1+ \frac{1}{a_n}\Big) \int_0^t ds \Big\{ \int \Gamma_{n,s}^{H, {\rm ex}} (\sqrt{f^n_s}) d \nu^n_s \Big\} + \frac{Cn^d}{a_n^2} \Big(1+ \frac{1}{a_n}\Big) \int_0^t \mathcal{H}_n (s) ds + o_n (1). 
\end{multline*}

Adding up the above estimates, for $n$ large enough such that $1+ 1/a_n < 2$, we have
\begin{align*}
	\frac{n^d}{a_n^2} \mathcal{H}_n (t) \leq &\frac{n^d}{a_n^2} \mathcal{H}_n (0) +   \int_0^t E_{\pi^n_s} \Big[ \frac{1}{a_n} \sum_{x \in \T_n^d} w_x  \Big( \big( \Delta_n + F^\prime (\rho_*) \big) \rho (s,\tfrac{x}{n}) - 2 \chi (\rho_*) \Delta_n H_s (\tfrac{x}{n}) \\
	&+ G (\rho_*)  H_s (\tfrac{x}{n}) - \partial_s \rho(s,\tfrac{x}{n})\Big) \Big] ds  + 2C \int_0^t \frac{n^d}{a_n^2} \mathcal{H}_n (s) ds + o_n (1).
\end{align*}
Since $\rho (t,u)$ is the solution to \eqref{hydrodynamic equation}, together with Lemmas \ref{lem: super-esponential estimate 1} and Lemma \ref{lem 1}, the above time integral has order $o_n (1)$.  We finally conclude the proof of Proposition \ref{pro: relative entropy} by noting that $\mathcal{H}_n (0) = 0$ and by using Gr{\" o}nwall's inequality. 

\bibliographystyle{plain}
\bibliography{bibliography.bib}

\begin{thebibliography}{10}

\bibitem{de1986reaction}
Anna De~Masi, Pablo~A Ferrari, and Joel~L Lebowitz.
\newblock Reaction-diffusion equations for interacting particle systems.
\newblock {\em Journal of statistical physics}, 44(3):589--644, 1986.

\bibitem{farfan2019static}
Jonathan Farfan, Claudio Landim, and Kenkichi Tsunoda.
\newblock Static large deviations for a reaction--diffusion model.
\newblock {\em Probability Theory and Related Fields}, 174:49--101, 2019.

\bibitem{gao2003moderate}
F.~Q. Gao and J~Quastel.
\newblock Moderate deviations from the hydrodynamic limit of the symmetric
  exclusion process.
\newblock {\em Science in China Series A: Mathematics}, 46(5):577--592, 2003.

\bibitem{gonccalves2024clt}
P~Gon{\c{c}}alves, M~Jara, R~Marinho, and O~Menezes.
\newblock {CLT} for {NESS} of a reaction-diffusion model.
\newblock {\em Probability Theory and Related Fields}, pages 1--41, 2024.

\bibitem{jara2018non}
M.~Jara and O.~Menezes.
\newblock Non-equilibrium fluctuations of interacting particle systems.
\newblock {\em arXiv preprint arXiv:1810.09526}, 2018.

\bibitem{jaram20nonequilireaction}
M.~Jara and O.~Menezes.
\newblock Non-equilibrium fluctuations for a reaction-diffusion model via
  relative entropy.
\newblock {\em Markov Process. Relat. Fields}, 26:95--124, 2020.

\bibitem{jona1993large}
G~Jona-Lasinio, C~Landim, and ME~Vares.
\newblock Large deviations for a reaction diffusion model.
\newblock {\em Probability theory and related fields}, 97:339--361, 1993.

\bibitem{klscaling}
C.~Kipnis and C.~Landim.
\newblock {\em Scaling limits of interacting particle systems}, volume 320.
\newblock Springer Science \& Business Media, 2013.

\bibitem{landim2018hydrostatics}
C.~Landim and K.~Tsunoda.
\newblock Hydrostatics and dynamical large deviations for a reaction-diffusion
  model.
\newblock In {\em Annales de l'Institut Henri Poincar{\'e}, Probabilit{\'e}s et
  Statistiques}, volume~54, pages 51--74. Institut Henri Poincar{\'e}, 2018.

\bibitem{wang2006moderate}
X.~Wang and F.~Q. Gao.
\newblock Moderate deviations from hydrodynamic limit of a {Ginzburg-Landau}
  model.
\newblock {\em Acta Mathematica Scientia}, 26(4):691--701, 2006.

\bibitem{xu2023equilibrium}
Lu~Xu and Linjie Zhao.
\newblock Equilibrium perturbations for stochastic interacting systems.
\newblock {\em Electronic Journal of Probability}, 28:1--30, 2023.

\bibitem{xue2024nonequilibrium}
X.~F. Xue.
\newblock Nonequilibrium moderate deviations from hydrodynamics of the simple
  symmetric exclusion process.
\newblock {\em Electronic Communications in Probability}, 29:1--16, 2024.

\bibitem{xue2021moderate}
X.~F. Xue and L.~J. Zhao.
\newblock Moderate deviations for the {SSEP} with a slow bond.
\newblock {\em Journal of Statistical Physics}, 182(3):48, 2021.

\bibitem{zhao2024moderate}
Linjie Zhao.
\newblock Moderate deviation principles for the {WASEP}.
\newblock {\em arXiv preprint arXiv:2405.16151}, 2024.

\end{thebibliography}
\end{document}